\documentclass[12pt]{amsart}
\usepackage{amssymb}
\usepackage{graphicx}
\usepackage{xcolor} 
\usepackage{tensor}
\usepackage{fullpage} 

\definecolor{green}{rgb}{0,0.8,0} 

\newtheorem{theorem}{Theorem}[section]

\newtheorem{lemma}[theorem]{Lemma}
\newtheorem{proposition}[theorem]{Proposition}
\theoremstyle{definition}

\theoremstyle{remark}
\newtheorem{remark}[theorem]{Remark}
\numberwithin{equation}{section}

\def\eps{\varepsilon }

\newcommand{\pa}{\partial}
\newcommand{\na}{\nabla}

\newcommand{\Del}{\Delta}

\newcommand{\vep}{\varepsilon}

\newcommand{\lam}{\lambda}

\newcommand{\bbr}{\mathbb R}

\newcommand{\beq}{\begin{equation}}
\newcommand{\eeq}{\end{equation}}

\newcommand{\be}{\begin{equation}}
\newcommand{\ee}{\end{equation}}
\newcommand{\ben}{\begin{equation*}}
\newcommand{\een}{\end{equation*}}
\newcommand{\la}{\label}
 \newcommand{\bmat}{\begin{pmatrix}} 
  \newcommand{\emat}{\end{pmatrix}}
    \newcommand{\tilu}{{\tilde{U}}}
        \newcommand{\tilq}{{\tilde{q}}}
            \newcommand{\tiln}{{\tilde{n}}}
      \newcommand{\s}{\sigma}
      \newcommand{\intr}{\int_\mathbb{R}}

      \newcommand{\el}{{\frac{\varepsilon}{\lambda}}}
            \newcommand{\eel}{{\frac{\varepsilon^2}{\lambda}}}

\newcommand{\intz}{{\int_0^1}}

\usepackage{verbatim}
\begin{document}
\bibliographystyle{plain}

\title{Contraction  for large perturbations of \\ traveling waves  in a hyperbolic-parabolic system\\ arising from a chemotaxis model}


\date{\today}%

\author[Choi]{Kyudong Choi}
\address[Kyudong Choi]{\newline Department of Mathematical Sciences, \newline Ulsan National Institute of Science and Technology, Ulsan 44919, Republic of Korea}
\email{kchoi@unist.ac.kr}

\author[Kang]{Moon-Jin Kang}
\address[Moon-Jin Kang]{\newline Department of Mathematics \& Research Institute of Natural Sciences, \newline Sookmyung Women's University, Seoul 04310,  Republic of Korea}
\email{moonjinkang@sookmyung.ac.kr}

\author[Kwon]{Young-Sam Kwon}
\address[Young-Sam Kwon]{\newline Department of Mathematics, \newline
Dong-A University,
Busan 49315,  Republic of Korea}
\email{ykwon@dau.ac.kr}

\author[Vasseur]{Alexis F. Vasseur}
\address[Alexis F. Vasseur]{\newline Department of Mathematics, \newline The University of Texas at Austin, Austin, TX 78712, USA}
\email{vasseur@math.utexas.edu}

\thanks{\textbf{Acknowledgement.}  The work of KC was partially supported by NRF-2018R1D1A1B07043065 and by the POSCO Science Fellowship of POSCO TJ Park Foundation. 
The work of MK was partially supported by NRF-2017R1C1B5076510. The work of YK was supported by Basic Science Research Program through the Na-
tional Research Foundation of Korea (NRF) funded by the Ministry of Education (NRF-2017R1D1A1B03030249).
The work of AV was partially supported by the NSF grant: DMS 1614918. }

\begin{abstract}
 We consider a hyperbolic-parabolic system arising from a chemotaxis model in angiogenesis, which is described by a Keller-Segel equation with singular sensitivity. It is known to allow viscous shocks (so-called traveling waves). We introduce a relative entropy of the system, which can capture how close a solution at a given time is to a given shock wave in  almost  $L^2$-sense. When the shock strength is small enough, we show the functional is non-increasing in time for any large initial perturbation.  The contraction property holds independently of the strength of the diffusion.
\end{abstract}

\maketitle

\section{Introduction and main theorem}

We consider the following one dimensional hyperbolic-parabolic system: 
\begin{align}\label{nq_nu} \begin{aligned}
\pa_t n -\pa_x(n q)& =\nu  \pa_{xx} n,\\
\pa_t q -\pa_xn & = 0 \quad \mbox{ for } x\in\mathbb{R} \quad\mbox{and for }t>0 
\end{aligned} \end{align}  where $\nu>0$ is a positive constant. 
We are interested in stability of viscous shocks (so-called traveling waves) of the above system.  
 
\subsection{Model from Chemotaxis}
 
The system \eqref{nq_nu} is related to  the following Keller-Segel system \cite{KSb}: 
 \begin{align}\label{KS} \begin{aligned}
\pa_t n - \nu\Del n& = - \na \cdot (n \chi(c) \na c ), \\
\pa_t c - \epsilon\Del c & = - c^m n  \quad \mbox{ for } \bold{x}\in\mathbb{R}^N \quad\mbox{and for }t>0 
\end{aligned} \end{align} with $m >0$ and $\epsilon\geq 0$. 
 In chemotaxis, the unknown  $n(\bold{x}, t)> 0 $ represents the bacterial density while
the unknown  $c(\bold{x}, t)> 0$ means the concentration of chemical nutrient consumed by bacteria at position $\bold{x},$ and time $t$.
We  assume that the given sensitivity function  $\chi(\cdot): \bbr^+ \to  \bbr^+ $ is  decreasing since the chemosensitivity gets usually lower as the  concentration of the chemical gets higher.  The positive constant
$m$ indicates  the consumption rate of nutrient $c$, and the non-negative constant $\epsilon\geq 0$ means  the chemical diffusion rate  for     $c$. \\

Such a Keller-Segel system can play a role of  a simplified model of    angiogenesis on
the formation of new blood vessels from pre-existing vessels, which is considered to be  the mechanism for tumor progression and metastasis  (see \cite{FoFrHu, FrTe,Le,Pe,  Rosen1, Sh}, and references therein).
In this interpretation, $n$ denotes the density of endothelial cells while $c$ does the concentration of  the protein known as the vascular endothelial  growth factor(VEGF). In biological implications, we usually  consider $\epsilon$ small (or negligible) (\textit{e.g.} see \cite{Le}).  \\ 

To derive our system \eqref{nq_nu}, we just take    $\chi(c)= c^{-1}$ and $m =1$,  $N=1$, 
 and $\epsilon=0$ into \eqref{KS} to get
\begin{align}\label{eq:main}
 \begin{aligned}
\pa_t n - \nu \pa_{xx} n& = - \pa_x \left(n \frac{ \pa_x c}{c} \right), \\
\pa_t c  & = - cn. 
\end{aligned} \end{align}
To have a traveling wave of \eqref{eq:main},  the chemosensitivity function $\chi( c)$  needs to be singular near  $c=0$  (\textit{e.g.} see \cite{KSb}). In particular, $\chi(c)= c^{-1}$ was assumed in \cite{KSb}. 
  Thanks to the 
    restriction   $m =1$, we can treat the singularity   in $c$ of the sensitivity by the Cole-Hopf transformation
\begin{equation*}
 q := - \pa_x [\ln c ]= -\frac{\pa_x c}{c}.
\end{equation*}   After the transform, we have 
\eqref{nq_nu} as in \cite{WaHi}.  

\subsection{Traveling waves of  \eqref{nq_nu}}
 We notice that if $n\geq0$, which is biologically relevant by the derivation from chemotaxis, then the principal part (\textit{i.e.} when $\nu=0$) of the system \eqref{nq_nu} is hyperbolic. By \cite{WaHi} (also see \cite{LiWa_siam}), it has been known that for any $\nu>0$,  \eqref{nq_nu} admits a smooth traveling wave $\bmat{\tiln}\\ {\tilq}\emat (x-\sigma t)$ connecting two end-states 
$(n_-,q_-)$ and $(n_+,q_+)$, \textit{i.e.},
\begin{equation}\label{bdry_cond}  {\tiln}(-\infty)=n_->0, \, \,\, \,{\tiln}(+\infty)=n_+>0,\, \,\,\,{\tilq}(-\infty)=q_-,\,\,\,\, {\tilq}(+\infty)=q_+
\end{equation} (we denote $\displaystyle \lim_{x\to \pm\infty} {f}(x)$ by  ${f}(\pm\infty)$ in short),
 provided the two end-states satisfy the Rankine-Hugoniot condition and the Lax entropy condition: 
\begin{align}
\begin{aligned}\label{end-con} 
&\exists~\sigma\in\bbr \mbox{ such that }~\left\{ \begin{array}{ll}
       -\sigma (n_+-n_-) - (n_+q_+ - n_-q_-) =0,\\
       -\sigma (q_+ -q_-) -(n_+-n_-)=0, \end{array} \right. \\
&\mbox{and either $n_->n_+$ and $q_-<q_+$ or  $n_-<n_+$ and $q_-<q_+$ holds.}        
\end{aligned}
\end{align} 
Here, the velocity $\sigma$ is given by
\be\la{sigma_eq}\sigma=\frac{-q_-\pm\sqrt{q_-^2+4n_+}}{2}.\ee 
More precisely, if $n_->n_+>0$, then $\sigma=\frac{-q_-+\sqrt{q_-^2+4n_+}}{2}>0$, whereas if $0<n_-<n_+$, then $\sigma=\frac{-q_--\sqrt{q_-^2+4n_+}}{2}<0$
 (See Subsection \ref{existence_traveling} for more details).  For this topic, we also refer  to the    survey paper \cite{Wa_survey} by Wang.\\  

In this parabolic conservation laws, it is an interesting topic to discuss how stable these viscous shocks are. By \cite{LiWa_siam}, it has been known that these waves are stable   if 
the anti-derivative of a perturbation $(n-\tiln, q-\tilq) $ is small in  the Sobolev space $[H^2(\bbr)]^2$. Thus the perturbation needs at least to have  the mean-zero condition: 
\ben
\exists x_0\in\bbr \mbox{ such that }\int_\bbr \bmat{ n_0(x)-\tiln(x-x_0)}\\ { q_0(x)-\tilq(x-x_0)}\emat  dx = \bmat{0}\\ {0}\emat. 
\een This condition is quite common in studying stability of viscous shocks since     \cite{Go} and \cite{KaMa}.\\

 In this paper, we introduce a relative entropy functional of the system, which plays a similar role of $L^2$-distance between a solution $(n,q)$ and a given shock profile $(\tiln, \tilq)$. Then  we show that the functional is non-increasing in time for any large initial perturbation.  Therefore, we prove that the contraction property holds independently of the size of the perturbation or the strength of the viscosity $\nu$. 
 It is remarkable that our result do not ask a perturbation to have either the mean-zero condition  or the smallness  in a Sobolev space. However,  we need that the shock strength $|n_--n_+|$ is small enough while this smallness on the wave amplitude was not required in \cite{LiWa_siam}. \\

   For the Cauchy problem of \eqref{nq_nu}, we refer to
\cite{GXZZ, LPZ, MWZ}  for global well-posedness. For multi-dimentional cases, see \cite{LLZ} and references therein. For stability of planar shocks under the mean-zero condition, we refer to 
\cite{CC_arxiv, CCKL}.

\begin{subsection}{Main result}\ \\

For $U_i=\begin{pmatrix} n_i   \\ q_i  \end{pmatrix}
$  with $n_i>0$ for $i=1,2$, we consider the relative entropy 
$$\eta(U_1 | U_2):= \frac{|q_1-q_2|^2}{2}+{\Pi}(n_1 | n_2),$$
 where
$$ {\Pi}(n_1 | n_2):= {\Pi}(n_1)-{\Pi}(n_2)-\nabla{\Pi}(n_2)(n_1-n_2),\qquad \Pi(n):=n\log n -n.$$
Since $\Pi(n)$ is strictly convex in $n$, its relative functional ${\Pi}(\cdot | \cdot)$ above is positive definite, and so is $\eta(\cdot | \cdot)$. That is,
  $\eta(U_1 | U_2)\geq 0$ for any $U_1$ and $U_2$, and 
$\eta(U_1 | U_2)= 0$ if and only if $U_1=U_2$.\\ 

Global existence and uniqueness of weak solutions to \eqref{nq_nu}   belonging to the   space 
\begin{align*}
\begin{aligned}\label{sp-T}
\mathcal{X}_T := \{ (n,q)\in L^\infty ((0,T)\times\bbr)^2~|~ n>0, ~  n^{-1}\in L^\infty((0,T)\times \bbr),~\partial_x n \in L^2((0,T)\times \bbr) \} 
\end{aligned}
\end{align*} for each $T>0$, is studied in \cite{CKV_in_preparation}. \\ 
 
 Here is the main result. We first state it for a fixed viscosity $\nu=1$. Then, in Remark \ref{rem_imp}, we illustrate that the main result still holds for any $\nu>0$.
 
\begin{theorem}\label{main_thm}
Let $\nu=1$. For a given constant state $(n_-,q_-)\in\bbr^+\times\bbr$, 
there exist  constants  $\delta_0\in(0,1/2)$ and $C>0$ such that the following is true:\\
For any $\eps,\lambda>0$ with $\eps\in(0,n_-)$ and $\delta_0^{-1}\eps<\lambda<\delta_0$, and for any $(n_+,q_+)\in\bbr^+\times\bbr$ satisfying \eqref{end-con} with $|n_--n_+|=\eps$, there exists  a smooth monotone function $a:\bbr\to\bbr^+$ with $\lim_{x\to\pm\infty} a(x)=1+a_{\pm}$ for some  constants $a_-, a_+$ with $|a_+-a_-|=\lambda$ such that the following holds:\\
Let $\tilu:=\bmat\tiln \\ \tilq\emat$ be a  traveling wave of \eqref{nq_nu} with the boundary condition \eqref{bdry_cond} and  with the speed $\s$ from \eqref{sigma_eq}.
For a given $T>0$, let $U(t,x):=\bmat n(t,x)\\ q(t,x) \emat$ be a solution to \eqref{nq_nu}  belonging to $\mathcal{X}_T$ with initial data $U_0(x):=\bmat n_0(x) \\ q_0(x) \emat $ satisfying \be
\label{initial_entropy} \int_{-\infty}^{\infty} \eta(U_0| \tilu) dx<\infty.\ee Then there exists an absolutely continuous shift function $X:[0,T]\rightarrow \mathbb{R}$ with 
$X\in W^{1,1}_{loc}$ and 
$X(0)=0$ such that   
\begin{align}
\begin{aligned}\la{ineq_contraction}
& \int_{-\infty}^{\infty} a(x-\s t) \eta\big(U(t,x-X(t))| \tilu(x-\s t)\big) dx \\
&\qquad + \delta_0\int_{0}^{t} \int_{-\infty}^{\infty} a(x-\s \tau) n\big(\tau,x-X(\tau)\big)\Big| \pa_x \Big(\log\frac{n(\tau,x-X(\tau))}{\tiln(x-\s \tau)}\Big)\Big|^2 dxd\tau\\
&\quad  \le \int_{-\infty}^{\infty} a(x) \eta\big(U_0(x)| \tilu(x)\big) dx,      
\end{aligned}
\end{align} 
and
\be\begin{split}\la{est_shift}
&|\dot X(t)-\sigma|\le \frac{1}{\vep^2}\Big(f(t) + C\int_{-\infty}^{\infty} \eta(U_0| \tilu) dx +1  \Big) \quad \mbox{ for \textit{a.e.} }t\in[0,T]\\
&\mbox{where $f$ is  some positive function  satisfying}\quad\|f\|_{L^1(0,T)} \le C\frac{\lambda}{\eps}\int_{-\infty}^{\infty} \eta(U_0| \tilu) dx.
\end{split}\ee 
 \end{theorem}

\begin{remark} The result can be considered to be an {\it a-priori} estimate for solutions of \eqref{nq_nu_1}. The existence issue of   solutions in the class $\mathcal{X}_T$ for any $T>0$ with the initial condition \eqref{initial_entropy} will be covered in the forthcoming paper \cite{CKV_in_preparation}.
The estimate on the dissipation in \eqref{ineq_contraction}, will be crucially used for the proof of the global existence of weak solutions to \eqref{nq_nu} in \cite{CKV_in_preparation}. 
\end{remark}

\begin{remark} 
Notice that it is enough to prove Theorem \ref{main_thm} in the case of $n_->n_+>0$. Indeed, the result for $n_+>n_->0$ is obtained by the change of variables $x\mapsto -x$ with $\s\mapsto -\s$. 
Therefore, from now on, we  assume $n_->n_+>0$ and thus 
\[
0<\sigma=\frac{-q_- +\sqrt{q_-^2+4n_+}}{2}.
\]
\end{remark}

\begin{remark}
Since the weight function $a$ satisfies that $|a(x)-1|\leq\lambda<\delta_0<1/2$ for all $x\in\bbr$, the contraction estimate \eqref{ineq_contraction} yields
\[
 \int_{-\infty}^{\infty} \eta\big(U(t,x-X(t))| \tilu(x-\s t)\big) dx   \le 4 \int_{-\infty}^{\infty} \eta\big(U_0(x)| \tilu(x)\big) dx.    
\]
\end{remark}

\begin{remark} \label{rem_imp}
In fact, such a contraction property \eqref{ineq_contraction} holds for any $\nu>0$, by scaling as follows. This scaling argument makes sense because of no condition on the strength of the initial perturbation.
Let $U^{\nu}$ and $\tilu^\nu$ be a solution and traveling wave to \eqref{nq_nu} with initial data $U_0$, respectively. Then, $U(t,x):=U^{\nu}(\nu t,\nu x)$ (resp. $\tilu(x):=\tilu^\nu(\nu x)$) is a solution (resp. traveling wave) to \eqref{nq_nu} with $\nu=1$. 
Therefore, using \eqref{ineq_contraction} together with the fact that
\begin{align*}
\begin{aligned}
&\int a^\nu(x-\s t) \eta\big(U^\nu (t,x-X^\nu (t))| \tilu^\nu(x-\s t)\big) dx \\
&\qquad\quad =\nu \int a(x-\s t/\nu) \eta\big(U (t/\nu ,x-X (t/\nu))| \tilu(x-\s t/\nu)\big) dx,
\end{aligned}
\end{align*} 
where $a^\nu(x):=a(x/\nu)$ and $X^\nu(t):=\nu X(t/\nu)$, we get
\[
\int a^\nu(x-\s t) \eta\big(U^\nu (t,x-X^\nu (t))| \tilu^\nu(x-\s t)\big) dx \le \int a^\nu(x) \eta\big(U_0 (x)| \tilu^\nu(x)\big) dx.
\]
 \end{remark}

 \end{subsection} 
 
\vspace{0.5cm}

{\bf Notations }
Throughout the paper, $C$ denotes a positive constant which may change from line to line, but which is independent of $\eps$ (the strength of the shock) and $\lambda$ (the total variation of the function $a$). 
The paper will consider two smallness conditions, one on $\eps$, and the other on $\eps/\lambda$. In the argument, $\eps$ will be far smaller than $\eps/\lambda$ .\\

\subsection{Ideas of Proof}\label{sub_idea}\

We basically take advantage of the new method introduced by Kang-Vasseur in
 \cite{KV_arxiv}, which is also used in the recent works \cite{Kang19, KV-unique19}. The main scenario of the method is briefly explained as follows.   \ \\
 
For a given viscous traveling wave $\tilu$ with small amplitude $|n_--n_+|=\eps$, the weight function $a$ is defined by $\tilu$ (see \eqref{def_a}). 
We employ the weighted relative entropy with the weight $a$, to get the contraction of any large perturbation $U$ from $\tilu$, up to a time-dependent shift $X(t)$. The shift function $X$ is constructed after the relative entropy computation in Lemma \ref{lem_relative}, which gives
\begin{eqnarray*}
&& \frac{d}{dt}\int_{-\infty}^{\infty} a(\xi) \eta\big(U(t,\xi+X(t))| \tilde U (\xi)\big) d\xi\\
 &&\qquad=\dot X(t) \mathcal{Y}(U(t,\cdot+X(t))) +\mathcal{I}^{bad}(U(t,\cdot+X(t)))- \mathcal{I}^{good}(U(t,\cdot+X(t))).
\end{eqnarray*}
Because of the relative entropy structure, the bad terms $\mathcal{I}^{bad}$ and the good terms $\mathcal{I}^{good}$ (i.e. $\mathcal{I}^{good}\ge 0$) are quadratic when the perturbation is small.  However, we have no uniform control on the size of the large perturbation $U(t,\cdot)$, therefore we should carefully estimate what happens for large values of $U(t,x)$. \\

The key idea of the technique is to exploit the degree of freedom of the shift $X(t)$ in the first term $\dot X(t) \mathcal{Y}(U(t,\cdot+X(t)))$. First of all, when $ \mathcal{Y}(U(t,\cdot))$ is not too small, we can construct  the shift $X(t)$ such that the term $\dot X(t) \mathcal{Y}(U(t,\cdot+X(t)))$ absorbs all the bad terms $\mathcal{I}^{bad}$  (see \eqref{def_shift}). Specifically, we ensure algebraically that the contraction holds as long as $|\mathcal{Y}(U(t))|\geq\eps^2$. Thus, the rest of the method is to show that the contraction still holds when $|\mathcal{Y}(U(t))|\leq \eps^2$. \\

In the argument, for the values of $t$ such that $|\mathcal{Y}(U(t))|\leq\eps^2$,  we construct the shift as a solution to the ODE: $\dot X(t)=-\mathcal{Y}(U(t,\cdot+X(t)))/\eps^4$. 
 From this point, we forget that $U$ is a solution to the system and $X(t)$ is the shift. That is, we leave out $X(t)$ and the $t$-variable of $U$. Therefore, it remains to show that for any function $U$ satisfying $\mathcal{Y}(U)\leq \eps^2$,
 \begin{equation*}\label{but}
 -\frac{1}{\eps^4}\mathcal{Y}^2(U) +\mathcal{I}^{bad}(U)- \mathcal{I}^{good}(U) \leq0.
 \end{equation*}
This is proved by Proposition \ref{main_prop} together with Lemma \ref{lem_max}.  This completes the proof of Theorem \ref{main_thm}. 
Proposition \ref{main_prop} is obtained thanks to a generic non-linear Poincar\'e type inequality (see Lemma \ref{lem_poincare}), which is first introduced in \cite{KV_arxiv}. It was first discovered for the scalar case in \cite{Kang-V-1}. The general method then follows \cite{KV_arxiv} by performing a careful expansion on the strength of the shock. Note that the parabolic system \eqref{nq_nu} is degenerate (that is, there is no diffusion in terms of $q$). Therefore, following \cite{KV_arxiv}, we first maximize the bad terms with respect to $q$ for $n$ fixed (see Lemma \ref{lem_max}). The expansion is then performed only on $n$. A new feature compared to \cite{KV_arxiv} is that the maximization can be performed only locally for $|n-\tiln|\ll 1$.\\

The remaining parts of the paper are organized as follows. In Section \ref{backg},
we introduce background materials including some properties of traveling waves, the definition of  the weight function $a(\cdot)$, and the main inequality (Lemma \ref{lem_relative})  from the relative entropy. Then in Section \ref{sec_3}, we give the definition of our shift $X$ and present the main proposition (Proposition \ref{main_prop}), which implies our main result (Theorem \ref{main_thm}). The proof of Proposition \ref{main_prop} is presented in Sections \ref{sec_near} and \ref{proof_main_prop}. In Section \ref{sec_near}, we get sharp estimates when $|n-\tiln|$ is small enough while in Section \ref{proof_main_prop}, we control all bad terms when  $|n-\tiln|$ is not small.

 \section{Background}\label{backg}
 \subsection{Moving frame}\la{moving_frame}
 From now on, we fix $\nu=1$ so our system \eqref{nq_nu} becomes
 \begin{align}\label{nq_nu_1} \begin{aligned}
\pa_t n -\pa_x(n q)& =  \pa_{xx} n, \\
\pa_t q -\pa_xn & = 0.
\end{aligned} \end{align}

For simplification of our analysis, we rewrite \eqref{nq_nu_1} into the following system, based on the change of variables $(t,x)\mapsto (t, \xi=x-\sigma t)$, where $\sigma=\frac{-q_- +\sqrt{q_-^2+4n_+}}{2}$:
 \begin{align}\label{nq} \begin{aligned}
\pa_t n -\sigma\pa_\xi n-\pa_\xi(n q)& =  \pa_{\xi\xi} n, \\
\pa_t q -\sigma\pa_\xi q -\pa_\xi n & = 0.
\end{aligned} \end{align} 
We are interested in a traveling wave solution $\tilu=\bmat\tiln \\ \tilq\emat$ of \eqref{nq_nu_1} as a solution of 
\begin{align}\label{NQ} \begin{aligned}
 -\sigma\pa_\xi {\tiln}-\pa_\xi({\tiln} {\tilq})& =  \pa_{\xi\xi} {\tiln}, \\
 -\sigma\pa_\xi {\tilq} -\pa_\xi{\tiln} & = 0.
\end{aligned} \end{align} 
\subsection{Existence and properties of traveling wave solutions}\label{existence_traveling}
 \ \\
In the sequel, without loss of generality, we consider the traveling wave $(\tilde n,\tilde q)$ satisfying $\tilde n(0)=\frac{n_-+n_+}{2}$. 
 \begin{lemma}\la{til_prop}
(1) For any $n_\pm, q_\pm$ with $n_->n_+>0$ satisfying \eqref{end-con}, the system \eqref{nq_nu_1} admits a smooth traveling wave $\bmat {\tiln} \\ {\tilq}\emat (x-\sigma t)$ 
connecting the two end-states $(n_-,q_-)$ and $(n_+,q_+)$ as \eqref{bdry_cond} with velocity
\be\la{sigma_e}\sigma=\frac{-q_-+\sqrt{q_-^2+4n_+}}{2}>0.\ee
Moreover,
\be\la{prop_til}\begin{split}
&{\tiln}'<0,\quad {\tilq}'=-\frac{{\tiln}'}{\sigma}>0, \mbox{ and}\\
&{\tiln}'=\frac{({\tiln}-n_-)({\tiln}-n_+)}{\sigma}.\\
\end{split}\ee 
(2) For any $(n_-, q_-)\in\bbr^+\times\bbr$, 
 there exist positive constants $\eps_1$ and $C$ such that for any $0<\eps<\eps_1$ and any $(n_+,q_+)\in\bbr^+\times\bbr$ satisfying \eqref{end-con} with $n_+=n_--\eps$, the following is true: \\
Let $\bmat {\tiln} \\ {\tilq}\emat (x-\sigma t)$  be the traveling wave connecting the two end states $(n_-,q_-)$ and $(n_+,q_+)$ such that ${\tiln}(0)=(n_-+n_+)/2$.\\
Then,
\be\la{tilde_estimate}\begin{split}
-\frac{\vep^2}{\s_-}e^{-\frac{\vep|\xi|}{\s_-}} \leq \tiln'(\xi)\leq - \frac{\vep^2}{4\s_-}e^{-\frac{\vep|\xi|}{\s_-}},
\end{split}\ee
where
\be\la{sigma_eq_-}\sigma_- :=\frac{-q_-+\sqrt{q_-^2+4n_-}}{2}.\ee
Moreover, we have 
\be\la{s_estimate}0<\frac{\s_-}{2}\leq(\s_--C\vep)\leq\s<\s_-, \ee
and 
\ben\label{sec-n}
|\tiln''(\xi)|\leq C\eps |\tiln'(\xi)|.
\een
\end{lemma}
\begin{proof}
$\bullet$ {\it proof of (1)} : 
The proof can be found in  \cite{LiWa_siam}  and \cite{WaHi}. Here we sketch its proof for completeness. 
Since 
$$\tiln''=-\s\tiln'-(\tiln \tilq)',$$
we have 
$$\tiln'=-\s(\tiln-n_-)-(\tiln\tilq-n_- q_-),$$
which can be written as
\[
\tiln'=-\s(\tiln-n_-)-\tiln(\tilq-q_-)-(\tiln-n_-)q_-.
\]
But, since $\tilq-q_- = -\frac{1}{\sigma}(\tiln-n_-)$ from $\tilq'=-\frac{1}{\sigma}\tiln'$, we have
\[
\frac{\tiln'}{\tiln-n_-}=-\s+\frac{\tiln}{\s} -q_-.
\] 
Since it follow from \eqref{sigma_e} that
\[
\sigma^2+q_-\sigma =n_+,
\]
we have
\[
\frac{\tiln'}{\tiln-n_-}=-\frac{\s^2+q_-\s -\tiln}{\s} =-\frac{n_+ -\tiln}{\s}. 
\]
That is,
\beq\label{st1}
{\tiln}'=\frac{({\tiln}-n_-)({\tiln}-n_+)}{\sigma}.
\eeq
This ODE has a smooth solution $\tiln$ connecting $n_-$ to $n_+$, and $\tiln'<0$. By $\tilq'=-\frac{1}{\sigma}\tiln'$ and \eqref{end-con}, we have $\tilq$.\\

$\bullet$ {\it proof of (2)} : 
First of all, since it follows from \eqref{sigma_e} and $n_+=n_--\eps$ that
$$\s=\frac{-q_-+\sqrt{q_-^2+4(n_--\vep)}}{2},$$ 
taking $\eps_1$ small enough such that 
\[
(\s_-/2)\leq(\s_--C\vep)\leq\s<\s_-,
\]
which gives \eqref{s_estimate}.\\
To show \eqref{tilde_estimate}, we first observe that \eqref{st1} yields
\beq\label{ode-n}
(\tiln -n_\pm)'=\frac{({\tiln}-n_-)({\tiln}-n_+)}{\sigma}.
\eeq
Since $\tiln'<0$ and ${\tiln}(0)=(n_-+n_+)/2$ imply
\begin{align}
\begin{aligned}\label{od}
&\xi\le0 ~\Rightarrow~n_--n_+ \ge \tilde n(\xi)-n_+\ge \tilde n(0)-n_+=\frac{n_--n_+}{2},\\
&\xi\ge0 ~\Rightarrow~n_--n_+ \ge n_--\tilde n(\xi)\ge n_-- \tilde n(0)=\frac{n_--n_+}{2}.
\end{aligned}
\end{align} 
it follows from \eqref{ode-n} and $n_--n_+=\eps$ that
\begin{align*}
\begin{aligned}
&\xi\le0 ~\Rightarrow~ -\frac{\eps}{2\s} (n_--{\tiln}) \le  (n_--\tiln)' \le -\frac{\eps}{\s} (n_--{\tiln}),\\
&\xi\ge0 ~\Rightarrow~-\frac{\eps}{\s} ({\tiln}-n_+) \le  (\tiln -n_+)' \le -\frac{\eps}{2\s} ({\tiln}-n_+).
\end{aligned}
\end{align*} 
These together with ${\tiln}(0)=(n_-+n_+)/2$ imply
\begin{align*}
\begin{aligned}
&\xi\le0 ~\Rightarrow~ \frac{\vep}{2}e^{-\frac{\vep|\xi|}{\s}}\leq(n_--\tiln)\leq\frac{\vep}{2}e^{-\frac{\vep|\xi|}{2\s}},\\
&\xi\ge0 ~\Rightarrow~\frac{\vep}{2}e^{-\frac{\vep\xi}{\s}}\leq(\tiln-n_+)\leq\frac{\vep}{2}e^{-\frac{\vep\xi}{2\s}}.
\end{aligned}
\end{align*} 
Applying the above estimates to \eqref{ode-n} together with \eqref{od}, we have 
$$-\frac{\vep^2}{2\s}e^{-\frac{\vep|\xi|}{2\s}} \leq \tiln'(\xi)\leq - \frac{\vep^2}{4\s}e^{-\frac{\vep|\xi|}{\s}}.$$
Finally, using \eqref{s_estimate}, we have the desired estimates in \eqref{tilde_estimate}.

Moreover, we differentiate \eqref{st1} to get
$$\tiln''=\tiln'\Big(\frac{\tiln-n_-}{\s}+\frac{\tiln-n_+}{\s}\Big). $$
Since
\[
\Big|\frac{\tiln-n_-}{\s}+\frac{\tiln-n_+}{\s}\Big|\leq \frac{2\eps}{\s} \le \frac{4\eps}{\s_-},
\]
we have
\[
|\tiln''(\xi)|\leq \frac{4\vep}{\s_-}|\tiln'(\xi)|.
\]
\end{proof}

\subsection{Definition of the weight $a$}
\ \\
For a given stationary solution  ${\tilu}$ of \eqref{nq}(\textit{i.e.} a solution of \eqref{NQ}), 
we define  $a(\cdot)$ by 
\be\la{def_a}
a:=1+\frac{\lambda}{\vep}(n_--\tiln).
\ee
Note 
\be\la{prop_a}
a(-\infty)=1, a(+\infty)=1+\lambda,\mbox{ and }a'=\Big(-\frac{\lambda}{\vep}\Big)\tiln'>0 \mbox{ by } \eqref{prop_til}.
\ee

\subsection{Relative entropy method}
As mentioned in Subsection \ref{sub_idea}, we employ the new analysis in \cite{KV_arxiv}, which is based on the relative entropy method. The method  is purely nonlinear, and allows to handle rough and large perturbations. The relative entropy method was first introduced by Dafermos \cite{Dafermos1} and Diperna \cite{DiPerna} to prove the $L^2$ stability and uniqueness of Lipschitz solutions to the hyperbolic conservation laws endowed with a convex entropy.\\ Recently, the relative entropy method has been extensively used in studying on the contraction (or stability) of large perturbations of viscous (or inviscid) shock waves (see \cite{CV,Kang19,Kang18_KRM,KV_arxiv,KVARMA,Kang-V-1,KVW,Leger,Serre-Vasseur,SV_16,SV_16dcds,Vasseur-2013,VW}).

To use the relative entropy method, we rewrite \eqref{nq} into the following general system of viscous conservation laws:
\be\la{U}
\pa_t U + \pa_{\xi}[A(U)]=\pa_\xi\Big[M(U)\pa_{\xi}\nabla\eta(U)\Big],
\ee
where 
\begin{align}
\begin{aligned}\label{def-flux}
&U:=\bmat n \\ q\emat, \quad A(U):=\bmat -nq -\sigma n \\ -n -\sigma q \emat ,\quad M(U):=\bmat n & 0 \\ 0 & 0 \emat,\\
&\eta(U):=\frac{|q|^2}{2}+{\Pi}(n)\quad \mbox{ where}\quad{\Pi}(n):=n\log n -n.
\end{aligned}
\end{align}
Indeed, since 
\beq\label{der-eta}
\nabla\eta(U):=\bmat \pa_n\eta(U)& \pa_q\eta(U)\emat=\bmat\log n & q \emat,
\eeq 
we see that \eqref{nq} is equivalent to \eqref{U}.\\
Notice that $\eta$ is a strictly convex entropy of the system \eqref{U}, since 
$$G(U):=-q n \log n -\sigma \eta(U)$$
is the entropy flux of $\eta$ such that $\partial_{i}  G (U) = \sum_{k=1}^{2}\partial_{k} \eta(U) \partial_{i}  A_{k} (U),\quad 1\le i\le 2$.\\

In general, for a given function $f$, we define its relative function $f(\cdot | \cdot)$ of two variables by
$$f(u|v):=f(u)-f(v)-\nabla f(v)(u-v).$$ Then for $U_i=\begin{pmatrix} n_i   \\ q_i  \end{pmatrix}, i=1,2$,
\beq\label{rel-A}
A(U_1|U_2)=A(U_1)-A(U_2)-\nabla A(U_2)(U_1-U_2)=
\bmat -(n_1-n_2)(q_1-q_2) \\ 0 \emat,
\eeq
and 
$$\eta(U_1|U_2)=\eta(U_1)-\eta(U_2)-\nabla\eta(U_2)(U_1-U_2)=
 \frac{|q_1-q_2|^2}{2}+{\Pi}(n_1 | n_2),$$
 { 
where 
\[
{\Pi}(n_1 | n_2)= {\Pi}(n_1)-{\Pi}(n_2)-\nabla{\Pi}(n_2)(n_1-n_2).
\]
Since ${\Pi}(n)=n\log n -n$, we find that 
\beq \label{def-pi}
{\Pi}(n_1 | n_2)= n_1\log(\frac{n_1}{n_2})-(n_1-n_2).
 \eeq
}
 
 We define the corresponding flux $G(\cdot;\cdot)$ for our relative entropy $\eta(\cdot|\cdot)$ by
\begin{align}
\begin{aligned}\label{rel-G}
G(U_1;U_2):&=G(U_1)-G(U_2)-\nabla\eta(U_2)(A(U_1)-A(U_2))\\
&= -(q_1-q_2){\Pi}(n_1| n_2)
-q_2{\Pi}(n_1| n_2)-(n_1-n_2)(q_1-q_2)-\s\eta(U_1|U_2).
\end{aligned}
\end{align}





In what follows, we use a simple notation: for any function $f:\mathbb{R}_{\geq0}\times\mathbb{R}\rightarrow\mathbb{R}$ and any shift $X:[0,\infty)\rightarrow \mathbb{R}$, 
 $$f^{\pm X}(t,\xi):=f(t,\xi\pm X(t)).$$
 
 { 
We also  introduce the function space
\be\label{def_H_space}
\mathcal{H}:=\{(m,p)\in  L^\infty(\bbr)\times L^\infty(\bbr) ~|~  m>0, m^{-1} \in L^{\infty}(\bbr),~  \pa_\xi \Big(\log\frac{m}{\tiln}\Big) \in L^2( \bbr) \}.
\ee \begin{remark}\label{rem_H}
As mentioned before, we consider the solution $U$ to \eqref{nq_nu} belonging to $\mathcal{X}_T$. Then, since $\partial_\xi n \in L^2((0,T)\times\bbr)$ and $n^{-1}\in L^\infty((0,T)\times\bbr)$, using $\tilde n\in L^\infty(\bbr)$ and $\tilde n' \in L^2(\bbr)$, we find 
\[
\pa_\xi \Big(\log\frac{n}{\tiln}\Big) \in L^2((0,T)\times\bbr),
\] 
which implies $U(t)\in \mathcal{H}$ for \textit{a.e.} $t\in[0,T]$. 
\end{remark}
}
\ \\

\begin{lemma}\la{lem_relative}
Let $\tilde U:={\tilde n \choose \tilde q}$ be the traveling wave in \eqref{NQ}, and $a:\bbr\to\bbr^+$ be the weight function by \eqref{def_a}. For any solution $U={  n \choose q  }
\in\mathcal{X}_T$ of \eqref{nq} 
for some $T>0$ and for any absolutely continuous shift $X:[0,T]\rightarrow \bbr$, we have, for \textit{a.e.} $t\in[0,T]$,
\begin{align}
\begin{aligned}\label{eq_rel1}
\frac{d}{dt}\int_{\bbr} a(\xi)\eta(U^X(t,\xi)|\tilde U(\xi)) d\xi =\dot X(t) \mathcal{Y}(U^X) +\mathcal{I}^{bad}(U^X)-\mathcal{I}^{good}(U^X),
\end{aligned}
\end{align}
where
\begin{align}
\begin{aligned}\label{badgood}
&\mathcal{Y}(U):= -\int_\bbr a'\eta(U|\tilde U) d\xi +\int_\bbr a\partial_\xi\nabla\eta(\tilde U) (U-\tilde U) d\xi,\\
&\mathcal{I}^{bad}(U):= -{\intr}\Big[a'{\Pi}(n|\tiln) +\Big( a'-a\frac{\tiln'}{\tiln} \Big) (n-\tiln)\Big](q-\tilq) d\xi-{\intr} a'\tilq{\Pi}(n|\tiln)d\xi  \\
&\qquad\qquad +\intr  \Big( a \frac{\tiln'}{\tiln} -a' \Big) n \Big(\log\frac{n}{\tiln}\Big)   \pa_\xi \Big(\log\frac{n}{\tiln}\Big) d\xi + \intr a \frac{\tiln''}{\tiln} \Pi(n|\tiln) d\xi,\\
&\mathcal{I}^{good}(U):=\s\intr a' \frac{|q-\tilq|^2}{2} d\xi+\s\intr a' {\Pi}(n | \tiln) d\xi  +\intr a n\Big| \pa_\xi \Big(\log\frac{n}{\tiln}\Big)\Big|^2 d\xi.
\end{aligned}
\end{align}
\end{lemma}
 \begin{remark}
By Remark \ref{rem_H}, we know $U(t)\in \mathcal{H}$ for \textit{a.e.} $t\in[0,T]$. It makes the above functionals $\mathcal{Y}, \mathcal{I}^{bad}, \mathcal{I}^{good}$ in \eqref{badgood}  well-defined  for \textit{a.e.} $t\in[0,T]$. 
\end{remark}
\begin{proof}
To derive the desired structure, we  use here a change of variables $\xi\mapsto \xi-X(t)$ as
\beq\label{move-X}
\int_{\bbr} a(\xi)\eta(U^X(t,\xi)|\tilde U(\xi)) d\xi=\int_{\bbr} a^{-X}(\xi)\eta(U(t,\xi)|\tilde U^{-X}(\xi)) d\xi.
\eeq
Then, by a straightforward computation together with \cite[Lemma 4]{Vasseur_Book} and the identity $G(U;V)=G(U|V)-\nabla\eta(V)A(U|V)$ (see also \cite{KV_arxiv}), we have
\begin{align*}
\begin{aligned}
&\frac{d}{dt}\int_{\bbr} a^{-X}(\xi)\eta(U(t,\xi)|\tilde U^{-X}(\xi)) d\xi\\
&=-\dot X \int_{\bbr} \!a'^{-X} \eta(U|\tilde U^{-X} ) d\xi +\int_\bbr \!\!a^{-X}\bigg[\Big(\nabla\eta(U)-\nabla\eta(\tilde U^{-X})\Big)\!\Big(\!\!\!-\partial_\xi A(U)+ \partial_\xi\Big(M(U)\partial_\xi\nabla\eta(U) \Big) \Big)\\
&\qquad -\nabla^2\eta(\tilde U^{-X}) (U-\tilde U^{-X}) \Big(-\dot X \partial_\xi\tilde U^{-X} -\partial_\xi A(\tilde U^{-X})+ \partial_\xi\Big(M(\tilde U^{-X})\partial_\xi\nabla\eta(\tilde U^{-X}) \Big)\Big)  \bigg] d\xi\\
&\quad =\dot X \Big( -\int_\bbr a'^{-X}\eta(U|\tilde U^{-X}) d\xi +\int_\bbr a^{-X}\partial_\xi\nabla\eta(\tilde U^{-X}) (U-\tilde U^{-X}) d\xi \Big) +I_1+I_2+I_3,
\end{aligned}
\end{align*}
where 
\begin{align*}
\begin{aligned}
&I_1:=-\int_\bbr a^{-X} \partial_\xi G(U;\tilde U^{-X}) d\xi,\\
&I_2:=- \int_\bbr a^{-X} \partial_\xi \nabla\eta(\tilde U^{-X}) A(U|\tilde U^{-X}) d\xi,\\
&I_3:=\int_\bbr a^{-X} \Big( \nabla\eta(U)-\nabla\eta(\tilde U^{-X})\Big) \partial_\xi \Big(M(U) \partial_\xi \nabla\eta(U) \Big)  d\xi \\
&\qquad -\int_\bbr a^{-X}\nabla^2\eta(\tilde U^{-X}) (U-\tilde U^{-X})\partial_\xi\Big(M(\tilde U^{-X})\partial_\xi\nabla\eta(\tilde U^{-X})\Big) d\xi.
\end{aligned}
\end{align*}
We first use \eqref{rel-A} and \eqref{rel-G} to have
\begin{align*}
\begin{aligned}
I_1&=\int_\bbr (a')^{-X} G(U;\tilde U^{-X}) d\xi\\
&=-{\intr} (a')^{-X}(q-\tilq^{-X}){\Pi}(n|\tiln^{-X})d\xi-{\intr} (a')^{-X}\tilq^{-X}{\Pi}(n|\tiln^{-X})d\xi\\
&\quad -\intr(a')^{-X}(n-\tiln^{-X})(q-\tilq^{-X})d\xi -\s\intr (a')^{-X}\eta(U|\tilu^{-X})d\xi,\\
I_2&=\int_\bbr a^{-X}\frac{(\tiln')^{-X}}{\tiln^{-X}}(n-\tiln^{-X})(q-\tilq^{-X})d\xi.
\end{aligned}
\end{align*}
For the parabolic part $I_3$, we rewrite it into
\begin{align*}
\begin{aligned}
I_3&=\intr a^{-X}\Big[\nabla \eta(U)-\nabla\eta(\tilu^{-X})\Big]\pa_\xi\Big[M(U)\pa_{\xi}[\nabla\eta(U)-\nabla\eta(\tilu^{-X})]\Big]d\xi\\
&\quad+\intr a^{-X}\Big[\nabla \eta(U)-\nabla\eta(\tilu^{-X})\Big]\pa_\xi\Big[M(U)\pa_{\xi}\nabla\eta(\tilu^{-X}) \Big]d\xi\\
&\quad -\int_\bbr a^{-X}\nabla^2\eta(\tilde U^{-X}) (U-\tilde U^{-X})\partial_\xi\Big[M(\tilde U^{-X})\partial_\xi\nabla\eta(\tilde U^{-X})\Big] d\xi\\
&=-\intr a^{-X}\pa_\xi\Big[\nabla \eta(U)-\nabla\eta(\tilu^{-X})\Big]\Big[M(U)\pa_{\xi}[\nabla\eta(U)-\nabla\eta(\tilu^{-X})]\Big]d\xi\\
&\quad-\intr (a')^{-X}\Big[\nabla \eta(U)-\nabla\eta(\tilu^{-X})\Big]\Big[M(U)\pa_{\xi}[\nabla\eta(U)-\nabla\eta(\tilu^{-X})]\Big]d\xi\\
&\quad+\intr a^{-X}\Big[\nabla \eta(U)-\nabla\eta(\tilu^{-X})\Big]\pa_\xi\Big[M(U)\pa_{\xi}\nabla\eta(\tilu^{-X}) \Big]d\xi\\
&\quad -\int_\bbr a^{-X}\nabla^2\eta(\tilde U^{-X}) (U-\tilde U^{-X})\partial_\xi\Big[M(\tilde U^{-X})\partial_\xi\nabla\eta(\tilde U^{-X})\Big] d\xi\\
&=: I_{31}+I_{32}+I_{33}+I_{34}.
\end{aligned}
\end{align*}
Substituting the explicit quantities in \eqref{def-flux}, we have
\begin{align*}
\begin{aligned}
I_{31}&= -\intr a^{-X} n\Big| \pa_\xi \Big(\log\frac{n}{\tiln^{-X}}\Big)\Big|^2 d\xi,\\
I_{32}&=-\intr (a')^{-X} n \Big(\log\frac{n}{\tiln^{-X}}\Big)   \pa_\xi \Big(\log\frac{n}{\tiln^{-X}}\Big)d\xi,\\
I_{33}&=\intr a^{-X}  \Big(\log\frac{n}{\tiln^{-X}}\Big)   \pa_\xi \Big(n \partial_\xi \log \tiln^{-X} \Big)d\xi,\\
I_{34}&=-\intr a^{-X}  \frac{\pa_\xi(\tiln^{-X}\pa_\xi \log\tiln^{-X})}{\tiln^{-X}} (n-\tiln^{-X}) d\xi = -\intr a^{-X}  \frac{(\tiln'')^{-X}}{\tiln^{-X}} (n-\tiln^{-X}) d\xi.
\end{aligned}
\end{align*}
Since
\begin{align*}
\begin{aligned}
I_{33}&=\intr a^{-X}  \Big(\log\frac{n}{\tiln^{-X}}\Big)   \pa_\xi \Big(n \partial_\xi \log \tiln^{-X} \Big)d\xi\\
&=\intr a^{-X}\Big(\log\frac{n}{\tiln^{-X}}\Big)\pa_\xi\Big(\frac{n}{\tiln^{-X}}(\tiln')^{-X}\Big)d\xi\\
&=\intr a^{-X}\Big(\log\frac{n}{\tiln^{-X}}\Big)\pa_\xi\Big(\frac{n}{\tiln^{-X}}\Big)(\tiln')^{-X} d\xi +\intr a^{-X}\Big(\log\frac{n}{\tiln^{-X}}\Big)\Big(\frac{n}{\tiln^{-X}}\Big)(\tiln'')^{-X} d\xi  \\
&=\intr a^{-X}\frac{(\tiln')^{-X}}{\tiln^{-X}}  n \Big(\log\frac{n}{\tiln^{-X}}\Big)   \pa_\xi \Big(\log\frac{n}{\tiln^{-X}}\Big) + \intr a^{-X} \frac{(\tiln'')^{-X}}{\tiln^{-X}} n \Big(\log\frac{n}{\tiln^{-X}}\Big)d\xi,
\end{aligned}
\end{align*}
we use \eqref{def-pi} to have
\begin{align*}
\begin{aligned}
I_{32}+I_{33}+I_{34}&=\intr  \Big( a^{-X} \frac{(\tiln')^{-X}}{\tiln^{-X}} -(a')^{-X} \Big) n \Big(\log\frac{n}{\tiln^{-X}}\Big)   \pa_\xi \Big(\log\frac{n}{\tiln^{-X}}\Big) d\xi\\
&\quad+ \intr a^{-X} \frac{(\tiln'')^{-X}}{\tiln^{-X}} \Pi(n|\tiln^{-X}) d\xi.
\end{aligned}
\end{align*}

Therefore, we have
\begin{align*}
\begin{aligned}
&\frac{d}{dt}\int_{\bbr} a^{-X}\eta(U|\tilde U^{-X}) d\xi\\
&\quad =\dot X \Big( -\int_\bbr a'^{-X}\eta(U|\tilde U^{-X}) d\xi +\int_\bbr a^{-X}\partial_\xi\nabla\eta(\tilde U^{-X}) (U-\tilde U^{-X}) d\xi  \Big)\\
&\qquad -{\intr} (a')^{-X}(q-\tilq^{-X}){\Pi}(n|\tiln^{-X})d\xi-{\intr} (a')^{-X}\tilq^{-X}{\Pi}(n|\tiln^{-X})d\xi  \\
&\qquad -\intr \Big( (a')^{-X}-a^{-X}\frac{(\tiln')^{-X}}{\tiln^{-X}} \Big) (n-\tiln^{-X})(q-\tilq^{-X})d\xi -\s\intr (a')^{-X}\eta(U|\tilu^{-X})d\xi \\
&\qquad +\intr  \Big( a^{-X} \frac{(\tiln')^{-X}}{\tiln^{-X}} -(a')^{-X} \Big) n \Big(\log\frac{n}{\tiln^{-X}}\Big)   \pa_\xi \Big(\log\frac{n}{\tiln^{-X}}\Big) d\xi \\
&\qquad + \intr a^{-X} \frac{(\tiln'')^{-X}}{\tiln^{-X}} \Pi(n|\tiln^{-X}) d\xi -\intr a^{-X} n\Big| \pa_\xi \Big(\log\frac{n}{\tiln^{-X}}\Big)\Big|^2 d\xi.
\end{aligned}
\end{align*}
Again, we use a change of variable $\xi\mapsto \xi+X(t)$ to have
\begin{align*}
\begin{aligned}
&\frac{d}{dt}\int_{\bbr} a\eta(U^X|\tilde U) d\xi\\
&\quad =\dot X \Big( -\int_\bbr a'\eta(U^X|\tilde U) d\xi +\int_\bbr a\partial_\xi\nabla\eta(\tilde U) (U^X-\tilde U) d\xi  \Big)\\
&\qquad -{\intr} a'(q^X-\tilq){\Pi}(n^X|\tiln)d\xi-{\intr} a'\tilq{\Pi}(n^X|\tiln)d\xi  \\
&\qquad -\intr \Big( a'-a\frac{\tiln'}{\tiln} \Big) (n^X-\tiln)(q^X-\tilq)d\xi -\s\intr a'\eta(U^X|\tilu)d\xi \\
&\qquad +\intr  \Big( a \frac{\tiln'}{\tiln} -a' \Big) n^X \Big(\log\frac{n^X}{\tiln}\Big)   \pa_\xi \Big(\log\frac{n^X}{\tiln}\Big) d\xi + \intr a \frac{\tiln''}{\tiln} \Pi(n^X|\tiln) d\xi\\
&\qquad -\intr a n^X\Big| \pa_\xi \Big(\log\frac{n^X}{\tiln}\Big)\Big|^2 d\xi.
\end{aligned}
\end{align*}
\end{proof}

\begin{remark}\label{rem:0}
Notice that since $\sigma>0$ and $ a' >0$, the three terms of $\mathcal{I}^{good}$ in \eqref{badgood} are non-negative. Therefore, $-\mathcal{I}^{good}$ consists of good terms, while $\mathcal{I}^{bad}$ consists of bad terms. 
\end{remark}

\subsection{Maximization in terms of $q-\tilq$}\la{subsec_def_again}
In order to estimate the right-hand side of \eqref{eq_rel1}, we will use Proposition \ref{prop_near} on a sharp estimate with respect to $n-\tiln$ when $|n-\tiln|\ll1$, for which we will first rewrite the functional $\mathcal{I}^{bad}$ in the right-hand side of \eqref{eq_rel1} into the maximized representation in terms of $q-\tilq$. More precisely, we use the first good term of  $\mathcal{I}^{good}  $  in \eqref{badgood}:
\[
-\s\intr a' \frac{|q-\tilq|^2}{2} d\xi,
\]
to separate $q-\tilq$ from the factors related to $n$ in the first term of $\mathcal{I}^{bad}  $ in  \eqref{badgood}. However, we will keep  $\mathcal{I}^{bad}$ for remaining cases as follows.

\begin{lemma}\la{lem_max}
Let $\tilde U:={\tilde n \choose \tilde q}$ be the traveling wave in \eqref{NQ}, and $a:\bbr\to\bbr^+$ be the weight function by \eqref{def_a}.
 Let $\delta$ be any positive constant.  Then, for any $U={  n \choose q  }\in\mathcal{H}$, we have
\begin{align}
\begin{aligned}\label{eq_relative}
\mathcal{I}^{bad}(U)-\mathcal{I}^{good}(U) =\mathcal{B}_\delta(U)-\mathcal{G}_\delta(U),
\end{aligned}
\end{align} 
where
\begin{align}
\begin{aligned}\label{bg-max}
&\varphi(n):= \frac{1}{\sigma}\Big(\Pi(n|\tiln) +\left(1+\frac{\eps}{\lambda} \frac{a}{\tiln} \right)(n-\tiln) \Big),\\
&\mathcal{B}_\delta(U):= -{\intr} a'\tilq{\Pi}(n|\tiln)d\xi +\frac{\sigma}{2} \intr a' |\varphi(n)|^2 {\mathbf 1}_{\{|(n/\tiln)-1| \le \delta\}} d\xi  \\
&\qquad\qquad -{\intr}a' \Big[{\Pi}(n|\tiln) + \Big( 1+ \el\frac{a}{\tiln}\Big) (n-\tiln)\Big](q-\tilq) {\mathbf 1}_{\{|(n/\tiln)-1| > \delta\}} d\xi\\
&\qquad\qquad -{\intr}a'\Big( 1+ \el\frac{a}{\tiln}\Big)n \Big(\log\frac{n}{\tiln}\Big)   \pa_\xi \Big(\log\frac{n}{\tiln}\Big) d\xi -\el{\intr}a''\frac{a}{\tiln} \Pi(n|\tiln) d\xi,\\
&\mathcal{G}_\delta(U):=\frac{\sigma }{2}\intr a' \Big(q-\tilq +\varphi(n) \Big)^2 {\mathbf 1}_{\{|(n/\tiln)-1| \le \delta\}}  d\xi  +\s\intr a' \frac{|q-\tilq|^2}{2}{\mathbf 1}_{\{|(n/\tiln)-1| > \delta\}}  d\xi\\
&\qquad\qquad +\s\intr a' {\Pi}(n | \tiln) d\xi  +\intr a n\Big| \pa_\xi \Big(\log\frac{n}{\tiln}\Big)\Big|^2 d\xi.
\end{aligned}
\end{align}
\end{lemma}
\begin{remark}
The bad term $\mathcal{B}_\delta( U)$ does not ask any information on   $q$ when   $|(n/\tiln)-1| \leq \delta$ .
\end{remark}
\begin{proof}
First of all, using $a'=-\frac{\lambda}{\vep}\tiln'$ and ${\tilq}'=-\frac{{\tiln}'}{\sigma}$, we have from \eqref{badgood} that
\begin{align*}
\begin{aligned}
&\mathcal{I}^{bad}(U):=\underbrace{ -{\intr}a' \Big[{\Pi}(n|\tiln) + \Big( 1+ \el\frac{a}{\tiln}\Big) (n-\tiln)\Big](q-\tilq) {\mathbf 1}_{\{|(n/\tiln)-1| \le \delta\}}  d\xi }_{=:J_1} \\
&\qquad\qquad  -{\intr}a' \Big[{\Pi}(n|\tiln) + \Big( 1+ \el\frac{a}{\tiln}\Big) (n-\tiln)\Big](q-\tilq) {\mathbf 1}_{\{|(n/\tiln)-1| > \delta\}}  d\xi  -{\intr} a'\tilq{\Pi}(n|\tiln)d\xi \\
&\qquad\qquad -{\intr}a'\Big( 1+ \el\frac{a}{\tiln}\Big) n \Big(\log\frac{n}{\tiln}\Big)   \pa_\xi \Big(\log\frac{n}{\tiln}\Big) d\xi -\el{\intr}a''\frac{a}{\tiln} \Pi(n|\tiln) d\xi, \\
&-\mathcal{I}^{good}(U):=\underbrace{ -\s\intr a' \frac{|q-\tilq|^2}{2}{\mathbf 1}_{\{|(n/\tiln)-1| \le \delta\}} d\xi}_{=:J_2}  -\s\intr a' \frac{|q-\tilq|^2}{2}{\mathbf 1}_{\{|(n/\tiln)-1| > \delta\}} d\xi \\
&\qquad\qquad -\s\intr a' {\Pi}(n | \tiln) d\xi  -\intr a n\Big| \pa_\xi \Big(\log\frac{n}{\tiln}\Big)\Big|^2 d\xi.
\end{aligned}
\end{align*}
By using a simple identity $\alpha x^2+ \beta x =\alpha(x+\frac{\beta}{2\alpha})^2-\frac{\beta^2}{4\alpha}$ with putting $x:=q-\tilq$, we have
\[
J_1+J_2 =\frac{\sigma}{2} \intr a' |\varphi(n)|^2 {\mathbf 1}_{\{|(n/\tiln)-1| \le \delta\}}  d\xi-\frac{\sigma }{2}\intr a' \Big(q-\tilq +\varphi(n) \Big)^2{\mathbf 1}_{\{|(n/\tiln)-1| \le \delta\}}  d\xi.
\]
Therefore, we have the desired relation.
\end{proof}

 \subsection{Global and local estimates on the relative quantity ${\Pi}(\cdot|\cdot)$}

\subsubsection{Global estimates on the relative quantity ${\Pi}(\cdot|\cdot)$} 

\begin{lemma}\la{cor_Q}
For given constants $\delta\in(0,\frac{1}{2}]$ and $n_->0$, there exist positive constants $C_1=C_1(n_-), C_2=C_2(n_-,\delta)$ and $C_3=C_3(n_-,\delta)$  such 
that the following inequalities hold:\\
1) For any $n_1>0$ and any $n_2>0$ with $\frac{n_-}{2}<n_2<n_-$,  
\be\la{Q_loc_} \frac{1}{{C_{1}}}|n_1-n_2|^2\leq  {\Pi}(n_1| n_2)\leq  {{C_{1}}}|n_1-n_2|^2 \quad \mbox{whenever }\, |\frac{n_1}{n_2}-1|\leq\delta,\ee 
\be\la{Q_global_} \frac{1}{{C_{2}}}(1+n_1\log^+ \frac{n_1}{n_2})\leq {\Pi}(n_1| n_2)\leq {{C_{2}}}(1+n_1\log^+ \frac{n_1}{n_2}) \quad \mbox{whenever }\, |\frac{n_1}{n_2}-1|\geq\delta,\ee 
\be\la{upper_1/2}\begin{split}
 & \frac{1}{{C_3}} |n_1-n_2|\leq {\Pi}(n_1|n_2) \le C_3|n_1-n_2|^2 \quad \mbox{whenever} \quad |\frac{n_1}{n_2}-1|\geq\delta,
 \end{split}\ee
 where $\log^+(y)$ is the positive part of $\log(y)$.\\

2) For any $n_1,n_2,m>0$ satisfying $m\le n_2 \le n_1$ or $n_1\le n_2 \le m$,
\beq\label{Pi-mono}
\Pi(n_1|m) \ge \Pi(n_2|m). 
\eeq

\end{lemma}
\begin{proof}
$\bullet$ {\it proof of \eqref{Q_loc_}} : 
We use the fact that the definition of the relative functional implies
\[
\Pi(n_1|n_2)=(n_1-n_2)^2\int_0^1\int_0^1 \Pi''(n_2 + st(n_1-n_2)) t dsdt.
\]
Notice that since $\Pi''(n)=1/n$,
\[
\Pi''(n_2 + st(n_1-n_2))=\frac{1}{st n_1 +(1-st)n_2},
\]
Since $|\frac{n_1}{n_2}-1|\leq\delta\le\frac{1}{2} $ and $\frac{n_-}{2}<n_2<n_-$, we have
\[
\frac{n_-}{4}<n_1<\frac{3n_-}{2}.
\]
Thus for any $0\le s,t\le 1$,
\[
\frac{1}{st \frac{3n_-}{2} +(1-st)n_-} \le \Pi''(n_2 + st(n_1-n_2)) \le \frac{1}{st  \frac{n_-}{4} +(1-st)  \frac{n_-}{2}}.
\]
Hence 
\[
c_1(n_1-n_2)^2 \le \Pi(n_1|n_2)\leq  c_2 (n_1-n_2)^2,
\]
where the constant $c_1, c_2$ only depends on $n_-$ as
\[
c_1=\int_0^1\int_0^1 \frac{t}{st \frac{3n_-}{2} +(1-st)n_-} dsdt,\quad c_2=\int_0^1\int_0^1 \frac{t}{st  \frac{n_-}{4} +(1-st)  \frac{n_-}{2}} dsdt.
\]

 {
$\bullet$ {\it proof of \eqref{Q_global_}} :
First of all, we observe from \eqref{def-pi} that
\beq\label{newdef-pi}
\Pi(n_1|n_2)=n_2\hat{{\Pi}}\Big(\frac{n_1}{n_2}\Big),\qquad \hat{{\Pi}}(y):=y\log y-(y-1)\quad \mbox{for }y>0.
\eeq
Notice that $\hat{{\Pi}} $ is smooth and non-negative on $(0,\infty)$, and  $\hat{\Pi}(y)=0$ if and only if $y=1$, since $\hat{{\Pi}} $ is strictly convex, and $y=1$ is the only critical point.\\
We will first estimate $\hat{{\Pi}}(y)$ as follows:\\
For any fixed $\delta\in(0,1/2]$, since $\hat{{\Pi}}'(y)=\log y <0$ for $0<y\le 1-\delta$, we have
\beq\label{inst-est0}
0<\hat{{\Pi}}(1-\delta) \le \hat{{\Pi}}(y)\le \lim_{s\rightarrow 0+}\hat{{\Pi}}(s)=1,\quad \forall y\in(0,1-\delta].
\eeq
On the other hand, using
$$\sup_{y\geq 1+\delta}\frac{y}{1+y\log y}<1,$$
we have a small constant $\kappa>0$ such that
\beq\label{inst-est1}
\kappa(1+y\log y)\le  \hat{{\Pi}}(y),\quad \forall y\ge 1+\delta.
\eeq
Moreover, since
\[
\hat{{\Pi}}(y) \le y\log y,\quad \forall y\ge 1+\delta,
\]
this together with \eqref{inst-est0} and \eqref{inst-est1} implies that 
there exists $C=C(\delta) >0$ such that
\[
\frac{1}{C }(1+y\log^+ y)\leq \hat{{\Pi}}(y)\leq {C }(1+y\log^+ y) \quad \mbox{for any } |y-1|\geq\delta.
\] 
Hence, this together with \eqref{newdef-pi} and $\frac{n_-}{2}<n_2<n_-$ implies \eqref{Q_global_}.
}

$\bullet$ {\it proof of \eqref{upper_1/2}} : 
Likewise, since there exists a constant $C=C(\delta) >0$ such that
 $$C^{-1} |y-1|\leq  \hat{\Pi}(y) \le C|y-1|^2 \quad \mbox{for any } |y-1|\geq\delta,$$
we have \eqref{upper_1/2}.

$\bullet$ {\it proof of \eqref{Pi-mono}} : 
Since $z\mapsto \Pi(z|y)$ is convex in $z>0$ and zero at $z=y$, $z\mapsto \Pi(z|y)$ is increasing in $|z-y|$, which implies \eqref{Pi-mono}.

\end{proof}
\begin{remark}\label{rem_c_1_indep}
$C_{1}$ is independent of $\delta\in(0,1/2]$ while $C_{2}$ blows up as $\delta$ goes to zero.
\end{remark}

\subsubsection{Local inequalities on the relative quantity ${\Pi}(\cdot|\cdot)$}
We  present now  some local estimates on ${\Pi}(n_1|n_2)$ for $|n_1- n_2|\ll 1$, based on  Taylor expansions. The specific  coefficients of the estimates will be crucially used in our local analysis.  
\begin{lemma}\label{lem:local}
For a given constant $n_->0$, there exist positive constants $C$ and $\delta_*$ such that for any $0<\delta<\delta_*$, the following is true.\\
For any $(n_1, n_2)\in \bbr_+^2$ satisfying $\Big|\frac{n_1}{n_2}-1\Big|<\delta$ and  $\frac{n_-}{2}<n_2<2n_-$, 
\beq\label{P3rd}
 \Pi(n_1|n_2) \ge \frac{n_2}{2}\Big[\Big(\frac{n_1}{n_2}-1\Big)^2-\frac{1}{3}\Big(\frac{n_1}{n_2}-1\Big)^3\Big].
\eeq
\beq\label{P2nd}
\Pi(n_1|n_2) \le \frac{n_2}{2}\Big[\Big(\frac{n_1}{n_2}-1\Big)^2-\frac{1}{3}\Big(\frac{n_1}{n_2}-1\Big)^3\Big]  +C\delta \Big|\frac{n_1}{n_2}-1\Big|^3.
\eeq
\end{lemma}
\begin{proof}
Since the function $\hat{\Pi}(y):=y\log y-(y-1)$ is smooth for $y>0$, we apply Taylor theorem to the function $\hat{\Pi}$. That is, using
\[
\hat{\Pi}'(y)=\log y,\quad \hat{\Pi}''(y)=\frac{1}{y},\quad \hat{\Pi}'''(y)=-\frac{1}{y^2},\quad \hat{\Pi}''''(y)=\frac{2}{y^3},
\]
for any $0<\delta<1$ and any $y\in[1-\delta,1+\delta]$, there exists $y_*$ between $1$ and $y$ such that
\[
\hat{\Pi}(y) =\frac{1}{2}(y-1)^2-\frac{1}{6}(y-1)^3+\frac{1}{12 }(y-1)^4+\hat{\Pi}^{(5)}(y_*) \frac{(y-1)^5}{5!}.
\]
Then we take $\delta_*$ small enough such that for any $0<\delta<\delta_*$ and $y\in[1-\delta,1+\delta]$, we have
\[
\frac{1}{2}(y-1)^2-\frac{1}{6}(y-1)^3 \le \hat{\Pi}(y) \le \frac{1}{2}(y-1)^2-\frac{1}{6}(y-1)^3 + C\delta |y-1|^3.
\]
Since $\Pi(n_1|n_2) =n_2\hat{\Pi}(\frac{n_1}{n_2})$, for any $(n_1, n_2)\in \bbr_+^2$ satisfying $\Big|\frac{n_1}{n_2}-1\Big|<\delta$, 
\[
\frac{n_2}{2}\Big[\Big(\frac{n_1}{n_2}-1\Big)^2-\frac{1}{3}\Big(\frac{n_1}{n_2}-1\Big)^3\Big] \le \Pi(n_1|n_2) \le \frac{n_2}{2}\Big[\Big(\frac{n_1}{n_2}-1\Big)^2-\frac{1}{3}\Big(\frac{n_1}{n_2}-1\Big)^3\Big]  +C\delta \Big|\frac{n_1}{n_2}-1\Big|^3,
\]
which completes the proof.
\end{proof}

\section{Proof of Theorem \ref{main_thm}}\label{sec_3}
Let $
 n_->0$ and  $ q_-\in\mathbb{R}$.   
 Consider $\lambda>0 $ 
and $\vep\in(0,n_-)$. Define $n_+>0$   by
$ \vep=(n_--n_+)$.
Let $\tilu:=\bmat\tiln \\ \tilq\emat$ be a  traveling wave of \eqref{nq} with the boundary condition \eqref{bdry_cond} and  with the speed $\s>0$ from \eqref{sigma_eq}. 
We define
 $a:\mathbb{R}\to\mathbb{R}_{>0}$ by \eqref{def_a}. 

\subsection{Construction of the shift $X$}\label{subsec_existence_shift}

For any fixed $\eps>0$, we consider a continuous function $\Phi_\eps$ defined by
\be\la{def_Phi}
\Phi_\eps (y)=
\left\{ \begin{array}{ll}
      \frac{1}{\varepsilon^2},\quad \mbox{if}~ y\le -\varepsilon^2, \\
      -\frac{1}{\varepsilon^4}y,\quad \mbox{if} ~ |y|\le \varepsilon^2, \\
       -\frac{1}{\varepsilon^2},\quad \mbox{if}  ~y\ge \varepsilon^2. \end{array} \right.
\ee
For a given solution $U\in\mathcal{X}_T$, we define a shift function $X(t)$ as the solution of the nonlinear ODE:
\be\la{def_shift}
\left\{ \begin{array}{ll}
        \dot X(t) = \Phi_\eps (\mathcal{Y}(U^X)) \Big(2|\mathcal{I}^{bad}(U^X)|+1 \Big)\quad\mbox{for \textit{a.e.} } t\in[0,T],\\
       X(0)=0,\end{array} \right.
\ee
where the functionals $\mathcal{Y}$ and $\mathcal{I}^{bad}$ are as in \eqref{badgood}.
 
Then, for any  solution $U\in\mathcal{X}_T$ for some $T>0$,  an absolutely continuous shift $X$ satisfying \eqref{def_shift} exists on $[0,T]$ and is unique.
Indeed, if we call the right-hand side of the ODE by $F(t,X)$, then it can be shown that
there exist functions $a,b\in L^2(0,T)$ such that
$$\sup_{x\in\bbr }|F(t,x)|\leq a(t)\quad \mbox{and}\quad
\sup_{x\in\bbr }|D_xF(t,x)|\leq b(t)
\quad \mbox{for } t\in[0,T] $$
  by using the information from $U\in\mathcal{X}_T$   together with the change of variables $\xi\mapsto \xi-X(t)$ as in \eqref{move-X}. Then we obtain the existence of a local solution by Picard's iteration argument, and it is extended up to time $T$ thanks to the estimate $a,b\in L^2(0,T)$. Uniqueness also follows (see Appendix \ref{appendix_shift} for the detail).\\

The following is the main proposition as a corner stone of proof of Theorem \ref{main_thm}.
\begin{proposition}\label{main_prop}
There exist $\delta_0\in(0,1/2)$ and $\delta_1\in(0,1/2)$ such that 
if positive constants $\vep$ and $\lam$ satisfy $\delta_0^{-1}\vep<\lam<\delta_0$, then
for any traveling wave $\tilde U:={\tilde n \choose \tilde q}$    in \eqref{NQ}  
 and
for any $U\in\mathcal{H}$ satisfying $|\mathcal{Y}(U)|\leq\vep^2$, we have
\be\la{est_main_prop}
\mathcal{R}(U):= -\frac{1}{\varepsilon^4}|\mathcal{Y}(U)|^2 + \mathcal{B}_{\delta_1}(U)+\delta_0\frac{\eps}{\lambda} |\mathcal{B}_{\delta_1}(U)|
-\mathcal{G}_{\delta_1}(U)+\delta_0 \mathcal{D}(U) \le 0,
\ee
where the functional $\mathcal{Y}$ is as in \eqref{badgood}, $\mathcal{B}_{\delta_1}$ and $\mathcal{G}_{\delta_1}$ are as in \eqref{bg-max}, and $
\mathcal{D}$ is defined by 
\begin{align} 
\begin{aligned}\label{d_only_a}
&\mathcal{D}(U):=  \intr a  n\Big| \pa_\xi \Big(\log\frac{n}{\tiln}\Big)\Big|^2 d\xi.
\end{aligned}
\end{align}  

\end{proposition}
 
We will first show how this proposition implies Theorem~\ref{main_thm}.

\subsection{Proof of Theorem~\ref{main_thm} from Proposition \ref{main_prop}}
\ \\ 

In order to prove the contraction \eqref{ineq_contraction} in Theorem \ref{main_thm}, by \eqref{eq_rel1} and \eqref{def_shift}, it is enough to show that for almost every   $t\in[0,T]$,
\ben
\Phi_\eps  (\mathcal{Y}(U^X)) \Big(2|\mathcal{I}^{bad}(U^X)|+1 \Big) \mathcal{Y}(U^X) +\mathcal{I}^{bad}(U^X)-\mathcal{I}^{good}(U^X)\le0.
\een 
For every $U\in \mathcal{H}$ we define 
\ben\label{rhs}
\mathcal{F}(U):=\Phi_\eps (\mathcal{Y}(U)) \Big(2|\mathcal{I}^{bad}(U)|+1 \Big)Y(U) +\mathcal{I}^{bad}(U)-\mathcal{I}^{good}(U).
\een
Since it follows from \eqref{def_Phi} that
\ben
\Phi_\eps (\mathcal{Y}) \Big(2|\mathcal{I}^{bad}|+1 \Big)\mathcal{Y}\le
\left\{ \begin{array}{ll}
     -2|\mathcal{I}^{bad}|,\quad \mbox{if}~  |\mathcal{Y}|\ge \varepsilon^2,\\
     -\frac{1}{\varepsilon^4}\mathcal{Y}^2,\quad  \mbox{if}~ |\mathcal{Y}|\le \varepsilon^2. \end{array} \right.
\een
we first find that for all $U\in \mathcal{H}$ satisfying  $|\mathcal{Y}(U)|\ge \eps^2 $, 
\[
\mathcal{F}(U) \le -|\mathcal{I}^{bad}(U)|-\mathcal{I}^{good}(U) \le 0.
\]

On the other hand, using \eqref{eq_relative}, we find that for any $\delta>0$ and any $U\in \mathcal{H}$ satisfying  $|\mathcal{Y}(U)|\le \eps^2 $, 
\[
\mathcal{F}(U) \le -\frac{1}{\varepsilon^4}\mathcal{Y}(U)^2 + \mathcal{B}_{\delta}(U)-\mathcal{G}_{\delta}(U).
\]
Then, Proposition \ref{main_prop} implies that for any $U\in \mathcal{H}$ satisfying  $|\mathcal{Y}(U)|\le \eps^2 $, 
$$
\mathcal{F}(U) \le -\delta_0\frac{\eps}{\lambda} |\mathcal{B}_{\delta_1}(U)| -\delta_0 \mathcal{D}(U)  \le 0.
$$
Therefore, using the above estimates with $U=U^X$ and $\delta_0<\frac{1}{2}$, we find that for a.e. $t\in[0,T]$,  
\begin{align}
\begin{aligned}\label{intemp}
&\frac{d}{dt}\int_{\bbr} a\eta(U^X|\tilu) d\xi +\delta_0 \mathcal{D}(U^X)= \mathcal{F}(U^X) + \delta_0 \mathcal{D}(U^X)\\
&\qquad \leq -|\mathcal{I}^{bad}(U^X)| {\mathbf 1}_{\{|\mathcal{Y}(U^X)|\ge \eps^2\}} -\delta_0\frac{\eps}{\lambda} |\mathcal{B}_{\delta_1}(U^X)| {\mathbf 1}_{\{|\mathcal{Y}(U^X)|\le \eps^2\}} \le 0,
\end{aligned}
\end{align}
which together with the initial condition $ \int_{\bbr} \eta(U_0|\tilu) d\xi <\infty$ yields that
\beq\label{tem_con}
\int_{\bbr} a\eta(U^X|\tilu) d\xi + \delta_0 \int_0^t \mathcal{D}(U^X) ds \le  \int_{\bbr} a \eta(U_0|\tilu) d\xi. 
\eeq

To conclude \eqref{ineq_contraction}, we recover $x$ variable from $\xi$ variable (see Subsection \ref{moving_frame}).\\
Hence  we have \eqref{ineq_contraction}  by redefining $X(t)$ by $(\s t-X(t))$.\\

Next, to estimate $|\dot X|$, we first observe that it follows from \eqref{def_Phi} and \eqref{def_shift} that
\be\label{est_shift0}
|\dot X|\le \frac{1}{\varepsilon^2}(2|\mathcal{I}^{bad}(U^X)|+1).
\ee 
Since \eqref{intemp} yields
\[
\frac{d}{dt}\int_{\bbr} a\eta(U^X|\tilu) d\xi + |\mathcal{I}^{bad}(U^X)| {\mathbf 1}_{\{|\mathcal{Y}(U^X)|\ge \eps^2\}} + \delta_0\frac{\eps}{\lambda} |\mathcal{B}_{\delta_1}(U^X)| {\mathbf 1}_{\{|\mathcal{Y}(U^X)|\le \eps^2\}}   \le 0,
\]
we have (using $ \|a\|_{L^\infty}\le 2$ by $\lambda< \delta_0<\frac{1}{2}$)
\beq\label{rem_est}
\int_0^{\infty}\Big( |\mathcal{I}^{bad}(U^X)| {\mathbf 1}_{\{|\mathcal{Y}(U^X)|\ge \eps^2\}} + \delta_0\frac{\eps}{\lambda} |\mathcal{B}_{\delta_1}(U^X)| {\mathbf 1}_{\{|\mathcal{Y}(U^X)|\le \eps^2\}} \Big)dt \le  2\int_{\bbr} \eta(U_0|\tilu) d\xi.
\eeq
Notice that \eqref{eq_relative} together with the definitions of $\mathcal{I}^{good}$ and $\mathcal{G}_{\delta_1}$ yields 
\begin{align*}
\begin{aligned}
&|\mathcal{I}^{bad}(U^X)|\\
&=|\mathcal{I}^{bad}(U^X)| {\mathbf 1}_{\{|\mathcal{Y}(U^X)|\ge \eps^2\}} +|\mathcal{I}^{bad}(U^X)| {\mathbf 1}_{\{|\mathcal{Y}(U^X)|\le \eps^2\}}  \\
&=|\mathcal{I}^{bad}(U^X)| {\mathbf 1}_{\{|\mathcal{Y}(U^X)|\ge \eps^2\}} +| \mathcal{I}^{good}(U^X) + \mathcal{B}_{\delta_1}(U^X)-\mathcal{G}_{\delta_1}(U^X)| {\mathbf 1}_{\{|\mathcal{Y}(U^X)|\le \eps^2\}}\\
&\le |\mathcal{I}^{bad}(U^X)| {\mathbf 1}_{\{|\mathcal{Y}(U^X)|\ge \eps^2\}}  +| \mathcal{B}_{\delta_1}(U^X)|{\mathbf 1}_{\{|\mathcal{Y}(U^X)|\le \eps^2\}}\\
&\quad + \frac{|\sigma| }{2}\intr |a'| \Big| \big(q^X-\tilq \big)^2 - \big(q^X-\tilq +\varphi(n^X)\big)^2 \Big| {\mathbf 1}_{\{|(n^X/\tiln)-1| \le \delta_1\}} d\xi\\
&\le |\mathcal{I}^{bad}(U^X)| {\mathbf 1}_{\{|\mathcal{Y}(U^X)|\ge \eps^2\}}  +| \mathcal{B}_{\delta_1}(U^X)|{\mathbf 1}_{\{|\mathcal{Y}(U^X)|\le \eps^2\}}\\
&\quad  +
C
\intr |a'| \Big( \big|q^X-\tilq \big|^2 + \Pi(n^X|\tilde n)^2 +|n^X-\tilde n|^2 \Big) {\mathbf 1}_{\{|(n^X/\tiln)-1| \le \delta_1\}} d\xi.
\end{aligned}
\end{align*}
Since \eqref{Q_loc_} implies that
\[
|(n/\tiln)-1| \le \delta_1 \quad \Rightarrow \quad \Pi(n|\tilde n) \le C_1|n-\tilde n|^2 \le C_1 (\delta_1 n_-)^2,
\]
we use \eqref{Q_loc_},   $a' \leq C \delta_0$, $\delta_0\le \frac{1}{2}\le a$ and by \eqref{tilde_estimate} and \eqref{prop_a}, to have
\[
|\mathcal{I}^{bad}(U^X)| \le |\mathcal{I}^{bad}(U^X)| {\mathbf 1}_{\{|\mathcal{Y}(U^X)|\ge \eps^2\}}  +| \mathcal{B}_{\delta_1}(U^X)|{\mathbf 1}_{\{|\mathcal{Y}(U^X)|\le \eps^2\}}+C\int_{\bbr} a\eta(U^X|\tilu) d\xi.
\] 
Therefore, it follows from \eqref{tem_con}, \eqref{est_shift0} and \eqref{rem_est} that
\ben\label{est_shift_proof1}
|\dot X|\le \frac{2}{\varepsilon^2} \Big( |\mathcal{I}^{bad}(U^X)| {\mathbf 1}_{\{|\mathcal{Y}(U^X)|\ge \eps^2\}}  +| \mathcal{B}_{\delta_1}(U^X)|{\mathbf 1}_{\{|\mathcal{Y}(U^X)|\le \eps^2\}} \Big)+\frac{C}{\varepsilon^2} \int_{\bbr} \eta(U_0|\tilu) d\xi +\frac{1}{\varepsilon^2},
\een
where
\ben\label{est_shift_proof2}
\int_0^{T}\Big( |\mathcal{I}^{bad}(U^X)| {\mathbf 1}_{\{|\mathcal{Y}(U^X)|\ge \eps^2\}} +  |\mathcal{B}_{\delta_1}(U^X)| {\mathbf 1}_{\{|\mathcal{Y}(U^X)|\le \eps^2\}} \Big)dt  \le \frac{2\lambda}{\delta_0\eps}\intr \eta(U_0|\tilu) d\xi.
\een  
Hence  we have \eqref{est_shift}  by redefining $X(t)$ by $(\s t-X(t))$ as mentioned above.\\

The remaining part is dedicated to prove  Proposition \ref{main_prop}. In Section \ref{sec_near}, we study behaviour of a scalar function in a certain class near a given traveling wave $\tiln$. Then, in Section \ref{proof_main_prop}, we construct a truncation $\bar{V}=\bmat \bar{m} \\ q\emat$ for   $V\in \mathcal{H}$ with   $|\mathcal{Y}(V)|\le \varepsilon^2 $ so that the truncated function $\bar{m}$ lies on the class covered in Proposition \ref{prop_near} while the error between $\bar{V}$ and $V$ in our functionals can be estimated in a proper way. It will give us  Proposition \ref{main_prop}.

\section{Estimates near the traveling wave}\la{sec_near}

\subsection{Expansion in the size of the traveling wave.}
We define the following functions:
\begin{align}
\begin{aligned}\label{badgood-n}
&\varphi(n):= \frac{1}{\sigma}\Big(\Pi(n|\tiln) +\left(1+\frac{\eps}{\lambda} \frac{a}{\tiln} \right)(n-\tiln) \Big),\\
&\mathcal{Y}_g(n):=-{\intr}a'\Big(\frac{|{{\varphi}}({n})|^2}{2}+{\Pi}(n|\tiln)\Big)d\xi -\frac{\vep}{\lambda} {\intr}aa'\Big(\frac{n-\tiln}{\tiln}+\frac{{{\varphi}}(n)}{\s}\Big)d\xi,\\
&\mathcal{I}_1(n):= -{\intr} a' \tilq {\Pi}(n|\tiln)d\xi  -\el{\intr}a''\frac{a}{\tiln} \Pi(n|\tiln) d\xi,\\
&\mathcal{I}_2(n):=\frac{\sigma}{2} \intr a' |\varphi(n)|^2 d\xi,\\
&\mathcal{G}_2(n):= \s\intr a' {\Pi}(n | \tiln) d\xi,\\
&\mathcal{D}(n):=\intr a n\Big| \pa_\xi \Big(\log\frac{n}{\tiln}\Big)\Big|^2 d\xi.
\end{aligned}
\end{align}

\begin{proposition}\label{prop_near}
For any $K>0$, there exist $\delta_{1}\in(0,\frac{1}{2})$ such that  for any $\delta_1^{-1}\eps<\lambda<\delta_1$ and for  any $\delta\in(0,\delta_{1})$,   
 the following is true:\\
For any function $n:\bbr\to \bbr^+$ such that  if
\be\label{condn_for_main_prop}
 \Big(|\mathcal{D}(n)|+\intr a'{\Pi}(n|\tiln)d\xi\Big)
 \mbox{ is finite},\quad
|\mathcal{Y}_g(n)|\leq K\frac{\varepsilon^2}{\lambda},\qquad \mbox{and } \quad\|\frac{n}{\tiln}-1\|_{L^\infty(\bbr)}\leq \delta_{1},
\ee
then
\begin{align*}
\begin{aligned}\label{ineq_taylor}
&\mathcal{R}_{\eps,\delta}(n):=-\frac{1}{\eps\delta}|\mathcal{Y}_g(n)|^2 +\left(\mathcal{I}_1(n)+\mathcal{I}_2(n)\right)+\delta\left(\frac{\eps}{\lambda}\right) \left(|\mathcal{I}_1(n)|+|\mathcal{I}_2(n)|\right)\\
&\qquad\qquad\qquad -\left(1-\delta\left(\frac{\eps}{\lambda}\right)\right)\mathcal{G}_2(n)-(1-\delta)\mathcal{D}(n)\le 0.
\end{aligned}
\end{align*}
\end{proposition}

To prove this proposition, we will use the nonlinear Poincar\'e type inequality in \cite{KV_arxiv}:
\begin{lemma}\label{lem_poincare}{[Proposition 3.3. in \cite{KV_arxiv}]}
For any given $M>0$, there exists $\delta^{*}=\delta^{*}(M)>0$, such that, for any $\delta\in(0,\delta^{*})$, the following is true:\\ For any $W\in L^2(0,1)$ with  
$\sqrt{y(1-y)}\partial_yW\in L^2(0,1)$, if $\int_0^1 |W(y)|^2\,dy\leq M$, then
\ben\la{inequal_poincare}
 \mathcal{R}_\delta(W)\leq0.
\een where  
\ben\la{def_poincare}\begin{split}
\mathcal{R}_\delta(W):&=-\frac{1}{\delta}\left(\int_0^1W^2\,dy+2\int_0^1 W\,dy\right)^2+(1+\delta)\int_0^1 W^2\,dy\\
&\quad+\frac{2}{3}\int_0^1 W^3\,dy 
+\delta \int_0^1 |W|^3\,dy 
-(1-\delta)\int_0^1 y(1-y)|\partial_y W|^2\,dy.
\end{split}\een
\end{lemma}

\subsubsection{Proof of Proposition \ref{prop_near}}
We first consider $\delta_1\in(0,\frac{1}{2})$ such that $\delta_1$ is smaller than
\beq\label{delta_1}
 \min(\sqrt{(n_-/2)}, \sqrt{\eps_1},\delta_*),
\eeq
where $\eps_1$ is as in Lemma \ref{til_prop}, and $\delta_*$ is as in Lemma \ref{lem:local}. Then it follows from Lemma \ref{til_prop} that
\beq\label{difs}
|\sigma-\sigma_-|\le C\eps, \quad \|\tilde n -n_-\|_{L^{\infty}(\bbr)}\le \eps, \quad \|\tilde q -q_-\|_{L^{\infty}(\bbr)}\le C\eps,
\eeq
where $\sigma_-$ denotes the constant in \eqref{sigma_eq_-}. \\
Note also that 
\beq\label{eps1}
\eps=\frac{\eps}{\lambda}\lambda < \delta_1^2.
\eeq

We define
\beq\label{def-y}
y(\xi):=\frac{n_--\tiln(\xi)}{\vep}\quad\mbox{ for } \quad\xi\in\mathbb{R}.
\eeq
Since $\tiln'(\xi)<0$, we will use a change of variables $\xi\in\bbr\mapsto y\in[0,1]$ to rewrite the functionals $\mathcal{Y}_g, \mathcal{I}_1, \mathcal{I}_2, \mathcal{I}_3, \mathcal{G}_2, \mathcal{D}$ in \eqref{badgood-n}.\\
Notice that it follow from \eqref{def_a} that $a=1+\lambda y$ and
\[
\frac{dy}{d\xi}=\frac{-1}{\vep}\tiln'(\xi)=\frac{1}{\lam}a'(\xi).
\]
In what follows, for simplification, we use the notation
\beq\label{wW}
w(y):=\frac{n(\xi(y))}{\tiln(\xi(y))}-1,\quad W(y):=\frac{\lam n_-}{\vep}w(y).
\eeq

$\bullet$ {\bf Change of variables for $\mathcal{Y}_g$:} 
We first set
\begin{align*}
\begin{aligned}
&\mathcal{Y}_g(n) =Y_1+Y_2+Y_3+Y_4,\\
&Y_1:= -{\intr}a' \frac{|{{\varphi}}({n})|^2}{2} d\xi,\\
&Y_2:=-{\intr}a' {\Pi}(n|\tiln) d\xi,\\
&Y_3:= -\frac{\vep}{\lambda} {\intr}aa' \frac{n-\tiln}{\tiln}d\xi,\\
&Y_4:=-\frac{\vep}{\lambda} {\intr}aa' \frac{{{\varphi}}(n)}{\s} d\xi.
\end{aligned}
\end{align*}
We use the change of variables with $|a-1|\le\delta_1$ to have
\beq\label{y3est}
\Big| Y_3 +  \eps \int_0^1 w\, dy \Big| \le  \eps\delta_1  \int_0^1 |w| dy
\eeq
Since 
\[
\varphi(n) = \frac{1}{\sigma}\Big[\Pi(n|\tiln) +\left(\tiln+\frac{\eps}{\lambda} a \right) \Big(\frac{n}{\tiln}-1\Big) \Big],
\]
it follows from \eqref{P3rd} and \eqref{P2nd} in Lemma \ref{lem:local} together with $\frac{\eps}{\lambda} a<2\delta_1$ that for any $n$ satisfying $\|\frac{n}{\tiln}-1\|_{L^\infty(\bbr)}\leq \delta_{1}$,
\begin{align*}
\begin{aligned}
&\Big| \varphi(n) -  \frac{\tiln}{\sigma} \Big(\frac{n}{\tiln}-1\Big) \Big| \le  C\delta_1  \Big|\frac{n}{\tiln}-1\Big|,\\
&\Big| |\varphi(n)|^2 - \left(\frac{\tiln}{\sigma}\right)^2 \Big(\frac{n}{\tiln}-1\Big)^2 \Big| \le  C\delta_1  \Big|\frac{n}{\tiln}-1\Big|^2.
\end{aligned}
\end{align*}
Then we use the change of variables to have
\begin{align*}
\begin{aligned}
&\Big| Y_4 +  \eps \int_0^1 \frac{\tiln}{\sigma^2} w\, dy \Big| \le  C\eps\delta_1  \int_0^1 |w| dy  ,\\
&\Big| Y_1 + \frac{\lambda}{2} \int_0^1 \left(\frac{\tiln}{\sigma}\right)^2 w^2\, dy \Big| \le  C\lambda\delta_1  \int_0^1 w^2 dy.
\end{aligned}
\end{align*}
Thus, using \eqref{difs} with \eqref{eps1}, we have
\begin{align}
\begin{aligned}\label{y14est}
&\Big| Y_4 +  \eps\frac{n_-}{\sigma_-^2} \int_0^1  w\, dy \Big| \le  C\eps\delta_1  \int_0^1 |w| dy  ,\\
&\Big| Y_1 + \frac{\lambda}{2} \left(\frac{n_-}{\sigma_-}\right)^2 \int_0^1  w^2\, dy \Big| \le  C\lambda\delta_1  \int_0^1 w^2 dy.
\end{aligned}
\end{align}
Likewise, since it follows from \eqref{P3rd} and \eqref{P2nd} that
\[
\Big| Y_2 + \frac{\lambda}{2} \int_0^1 \tiln w^2\, dy \Big| \le  C\lambda\delta_1  \int_0^1 w^2 dy,
\]
we have
\beq\label{y2est}
\Big| Y_2 + \frac{\lambda n_-}{2} \int_0^1 w^2\, dy \Big| \le  C\lambda\delta_1  \int_0^1 w^2 dy.
\eeq
Therefore, combining \eqref{y3est}, \eqref{y14est}, \eqref{y2est} with the notation \eqref{wW}, we have
\[
\left|  \mathcal{Y}_g + \frac{\eps^2}{2\lambda}\left(\frac{1}{n_-}+\frac{1}{\s_-^2}\right) \Big(\int_0^1 W^2 dy + 2\int_0^1 W dy\Big)  \right| \le C\delta_1\frac{\eps^2}{\lambda} \Big(\int_0^1 W^2 dy + \int_0^1 |W| dy\Big).
\]
Setting $\beta:=2\left(\frac{1}{n_-}+\frac{1}{\s_-^2}\right)^{-1}$, we have
\beq\label{YW}
\left| \beta\frac{\lambda}{\eps^2} \mathcal{Y}_g + \int_0^1 W^2 dy + 2\int_0^1 W dy  \right| \le C\delta_1 \Big(\int_0^1 W^2 dy + \int_0^1 |W| dy\Big).
\eeq

$\bullet$ {\bf Change of variables for $\mathcal{I}_1, \mathcal{I}_2$:} 
We first use \eqref{P3rd} and \eqref{P2nd} to find that for any $n$ satisfying $\|\frac{n}{\tiln}-1\|_{L^\infty(\bbr)}\leq \delta_{1}$,
\begin{align*}
\begin{aligned}
\mathcal{I}_1\le -  \intr a' \frac{\tilq \tiln}{2} \Big(\frac{n}{\tiln}-1\Big)^2 \, d\xi  +  \intr a' \frac{\tilq \tiln}{6} \Big(\frac{n}{\tiln}-1\Big)^3 \, d\xi + C\delta_1 \intr a' \Big|\frac{n}{\tiln}-1\Big|^3 \, d\xi
+ C\frac{\vep^2}{\lam} \intr a' \Big(\frac{n}{\tiln}-1\Big)^2 \, d\xi.
\end{aligned}
\end{align*}
Then using  \eqref{difs}, we have
\beq\label{fb1}
\mathcal{I}_1\le -\lambda \frac{q_- n_-}{2}  \int_0^1 w^2\, dy +\lambda \frac{q_- n_-}{6}  \int_0^1 w^3\, dy  +  C\lambda\eps  \int_0^1 w^2 dy+ C\delta_1\lambda  \int_0^1 |w|^3 dy.
\eeq
Since 
\[
|\varphi(n)|^2 = \frac{1}{\sigma^2}\Big[  \left(\tiln+\frac{\eps}{\lambda} a \right)^2 \Big(\frac{n}{\tiln}-1\Big)^2 +2 \left(\tiln+\frac{\eps}{\lambda} a \right) \Big(\frac{n}{\tiln}-1\Big) \Pi(n|\tiln) +\Pi(n|\tiln)^2 \Big],
\]
using \eqref{P2nd}, we have
\begin{align*}
\begin{aligned}
\mathcal{I}_2 &\le  \intr a' \frac{\tiln^2}{2\sigma} \Big(\frac{n}{\tiln}-1\Big)^2 \, d\xi + \frac{\eps}{\lambda}\intr a' a \frac{\tiln}{\sigma} \Big(\frac{n}{\tiln}-1\Big)^2 \, d\xi +  \intr a' \frac{\tiln^2}{2\sigma} \Big(\frac{n}{\tiln}-1\Big)^3 \, d\xi \\
 &\quad +C\delta_1\frac{\eps}{\lambda} \intr a'  \Big(\frac{n}{\tiln}-1\Big)^2 \, d\xi +  C\delta_1 \intr a' \Big|\frac{n}{\tiln}-1\Big|^3 \, d\xi.
\end{aligned}
\end{align*}
Thus,
\begin{align}
\begin{aligned}\label{fb2}
\mathcal{I}_2 &\le \lambda\frac{n_-^2}{2\s_-}   \int_0^1 w^2\, dy + \eps \frac{n_-}{\s_-}   \int_0^1 w^2\, dy + \lambda \frac{n_-^2}{2\s_-}   \int_0^1 w^3\, dy\\
&\quad+ C\delta_1 \eps\int_0^1 w^2\, dy+ C\delta_1\lambda \int_0^1 |w|^3\, dy.
\end{aligned}
\end{align}

$\bullet$ {\bf Change of variables for $\mathcal{G}_2$:} 
We use \eqref{P3rd} and \eqref{difs} to find that 
\begin{align}
\begin{aligned}\label{fg2}
\mathcal{G}_2 &\ge \s {\intr} a' \frac{\tiln}{2}  \left(\frac{n}{\tiln}-1\right)^2 d\xi -\s {\intr} a' \frac{\tiln}{6}  \left(\frac{n}{\tiln}-1\right)^3 d\xi  \\
&\ge \frac{\lambda\s_- n_-}{2} \int_0^1 w^2 dy - \frac{\lambda \s_-n_-}{6}  \int_0^1 w^3 dy -C\lambda\eps \int_0^1 w^2 dy - C\lambda\eps \int_0^1 |w|^3\, dy.
\end{aligned}
\end{align}

$\bullet$ {\bf Estimates on $\mathcal{I}_1+\mathcal{I}_2-\mathcal{G}_2$:} 
We combine \eqref{fb1}, \eqref{fb2} and  \eqref{fg2} to have
\begin{align*}
\begin{aligned}
&\mathcal{I}_1+\mathcal{I}_2-\mathcal{G}_2 \\
&\quad\le \frac{\lambda}{2}\frac{n_-}{\s_-} \underbrace{(n_- -q_-\s_- -\s_-^2)}_{=:J_1} \int_0^1 w^2 dy + \eps \frac{n_-}{\s_-}   \int_0^1 w^2\, dy \\
&\quad\quad+\frac{\lambda}{2}\frac{n_-}{\s_-} \underbrace{\left(n_- +\frac{1}{3}q_-\s_- +\frac{1}{3}\s_-^2 \right)}_{=:J_2} \int_0^1 w^3 dy + C\eps\delta_1 \int_0^1 w^2\, dy+ C\lambda\delta_1 \int_0^1 |w|^3\, dy.
\end{aligned}
\end{align*}
Since the constant $\sigma_-=\frac{-q_-+\sqrt{q_-^2+4n_-}}{2}$ solves the quadratic equation $\s_-^2 +q_-\s_- -n_-=0$, the above coefficients $J_1, J_2$ become $J_1=0$, $J_2=\frac{4}{3}n_-$. \\
Therefore, we have
\begin{align*}
\begin{aligned}
&\mathcal{I}_1+\mathcal{I}_2-\mathcal{G}_2 \\
&\quad\le \frac{\eps^3}{\lambda^2}\frac{1}{n_-\s_-} \Big( \int_0^1 W^2 dy + \frac{2}{3}   \int_0^1 W^3\, dy\Big) + C \frac{\eps^3}{\lambda^2} \delta_1  \Big( \int_0^1 W^2 dy +    \int_0^1 |W|^3\, dy\Big),
\end{aligned}
\end{align*}
which can be rewritten as (by normalizing the right-hand side above)
\begin{align}
\begin{aligned}\label{sum-bg}
& n_-\s_-\frac{\lambda^2}{\eps^3} \Big(\mathcal{I}_1+\mathcal{I}_2-\mathcal{G}_2\Big) \\
&\quad\le  \int_0^1 W^2 dy + \frac{2}{3}   \int_0^1 W^3\, dy + C  \delta_1  \Big( \int_0^1 W^2 dy + \int_0^1 |W|^3\, dy\Big).
\end{aligned}
\end{align}

As in  \eqref{fb1}, \eqref{fb2} and  \eqref{fg2}, 
 we can estimate 
\[
|\mathcal{I}_1|+|\mathcal{I}_2|+|\mathcal{G}_2| \le C\lambda \int_0^1 w^2 dy \le C\frac{\eps^2}{\lambda} \int_0^1 W^2 dy,
\]
which yields
\[
\delta_1\frac{\eps}{\lambda}  \Big(|\mathcal{I}_1|+|\mathcal{I}_2|+|\mathcal{G}_2|\Big) \le C\delta_1\frac{\eps^3}{\lambda^2} \int_0^1 W^2 dy.
\]
Therefore, we have
\beq\label{sum-a}
 n_-\s_-\frac{\lambda^2}{\eps^3}\Big[ \delta_1\frac{\eps}{\lambda}  \Big(|\mathcal{I}_1|+|\mathcal{I}_2|+|\mathcal{G}_2|\Big) \Big] \le C\delta_1 \int_0^1 W^2 dy.
\eeq

$\bullet$ {\bf Change of variables for $\mathcal{D}$:} 
Since
\[
\mathcal{D}=\intr a n\Big(\frac{\tiln}{n}\Big)^2 \Big| \pa_\xi \Big(\frac{n}{\tiln}-1\Big)\Big|^2 d\xi,
\]
we have
\[
\mathcal{D}=\int_0^1 
a
 \frac{\tiln^2}{n} |\partial_y w|^2 \Big(\frac{dy}{d\xi}\Big) dy.
\]
To compute $\frac{dy}{d\xi}$, using \eqref{def-y} with $n_- -n_+=\eps$, and \eqref{prop_til}, we have
\[
y(1-y)=\frac{(n_--\tiln)}{\vep}\frac{(\tiln-n_+)}{\vep}=-\frac{\s\tiln'}{\vep^2}=\frac{\s}{\vep\lam}a',
\]
which implies
\[
\frac{dy}{d\xi} = \frac{a'}{\lambda} =\frac{\eps}{\sigma} y(1-y).
\]
Since
\[
\inf_y\Big(a\frac{\tiln^2}{n}\Big)\geq 
\inf_y\Big(\frac{\tiln^2}{n}\Big)\geq
n_+(1-\delta_1)
\geq (n_--\vep)(1-\delta_1)\geq n_--C\delta_1
\]
using \eqref{difs}, we have
\begin{align*}
\begin{aligned}
\mathcal{D}
&\ge \eps \frac{n_-}{\s_-} (1-C\delta_1)  \int_0^1 y(1-y)  |\partial_y w|^2 dy.
\end{aligned}
\end{align*}
Therefore,
\[
\mathcal{D}\ge \frac{\eps^3}{\lambda^2} \frac{1}{n_-\s_-} (1-C\delta_1)  \int_0^1 y(1-y)  |\partial_y W|^2 dy.
\]
Hence
\beq\label{finald}
- n_-\s_-\frac{\lambda^2}{\eps^3} \mathcal{D} \le - (1-C\delta_1)  \int_0^1 y(1-y)  |\partial_y W|^2 dy.
\eeq

$\bullet$ {\bf A uniform bound of $\intz W^2dy $:} 
Using \eqref{condn_for_main_prop} and \eqref{YW}, we have
\ben\begin{split}
\intz W^2dy-2\Big|\intz Wdy\Big|&\leq\intz W^2dy+2\intz Wdy\\
&\leq\Big|\beta\frac{\lambda}{\eps^2}\mathcal{Y}_g+\intz W^2dy+2\intz Wdy\Big|+\beta\frac{\lambda}{\eps^2}\Big|\mathcal{Y}_g\Big|\\
&\leq C\delta_1 \Big(\int_0^1 W^2 dy + \int_0^1 |W| dy\Big)+\beta K,
\end{split}\een   where $K$ is the constant in the assumption \eqref{condn_for_main_prop}.\\
Using $$\Big|\intz Wdy\Big|\leq \intz |W|dy\leq \frac{1}{8}\intz W^2dy+2,$$
we have
\ben\begin{split}
\intz W^2dy
&\leq 2\Big|\intz Wdy\Big|+C\delta_1 \Big(\int_0^1 W^2 dy + \int_0^1 |W| dy\Big)+\beta K\\
&\leq \frac{1}{2}\intz W^2dy+C
\end{split}\een  by taking $\delta_1$ small enough.
Therefore there exists a positive constant $M$ depending on $K$ such that
\beq\label{controlW}
\intz W^2dy \le M.
\eeq

$\bullet$ {\bf Control on $-|\mathcal{Y}_g|^2$ :}
 As in \cite{KV_arxiv}, we here use the following inequality: For any $a,b\in \bbr$,  
 $$
 -a^2\leq -\frac{b^2}{2}+|b-a|^2.
 $$
 Using this inequality with 
 $$
 a=- \beta\frac{\lambda}{\eps^2}\mathcal{Y}_g ,\qquad b=\int_0^1W^2\,dy+2\int_0^1 W\,dy,
 $$
 we find
 \begin{eqnarray*}
 && - n_-\s_-\frac{\lambda^2}{\eps^3} \frac{|\mathcal{Y}_g|^2}{\eps \delta_1}
 =-\frac{n_-\s_-}{\delta_1\beta^2}\left| \beta\frac{\lambda}{\eps^2}\mathcal{Y}_g\right|^2 \\
 &&\qquad\leq -\frac{n_-\s_-}{2\delta_1\beta^2}\left| \int_0^1W^2\,dy+2\int_0^1 W\,dy\right|^2\\
  &&\qquad \qquad +\frac{n_-\s_-}{\delta_1\beta^2}\left| \beta\frac{\lambda}{\eps^2} \mathcal{Y}_g + \int_0^1 W^2 dy + 2\int_0^1 W dy  \right|^2.
 \end{eqnarray*}
Then by \eqref{YW}, we have
  \begin{eqnarray*}
  - n_-\s_-\frac{\lambda^2}{\eps^3} \frac{|\mathcal{Y}_g|^2}{\eps \delta_1}\leq -\frac{n_-\s_-}{2\delta_1\beta^2}\left| \int_0^1W^2\,dy+2\int_0^1 W\,dy\right|^2 +C\delta_1 \left(\int_0^1W^2\,dy+\int_0^1|W|\,dy \right)^2.
 \end{eqnarray*}
Using \eqref{controlW}, we have 
 $$
 \left(\int_0^1W^2\,dy+\int_0^1|W|\,dy \right)^2\leq \left(\int_0^1W^2\,dy+\sqrt{\int_0^1|W|^2\,dy }\right)^2\leq C\int_0^1W^2\,dy.
 $$
Therefore, we have
\beq\label{Y^2}
   - n_-\s_-\frac{\lambda^2}{\eps^3} \frac{|\mathcal{Y}_g|^2}{\eps \delta_1}\leq -\frac{n_-\s_-}{2\delta_1\beta^2}\left| \int_0^1W^2\,dy+2\int_0^1 W\,dy\right|^2 +C\delta_1 \int_0^1W^2\,dy.
\eeq

$\bullet$ {\bf Conclusion:}  
Since $\mathcal{G}_2\ge 0$, we see that for any $\delta<\delta_1$, 
\[
\mathcal{R}_{\eps,\delta}(n) \leq -\frac{1}{\eps\delta_1}|\mathcal{Y}_g|^2 +\left(\mathcal{I}_1+\mathcal{I}_2-\mathcal{G}_2 \right)+\delta_1\frac{\eps}{\lambda} \left(|\mathcal{I}_1|+|\mathcal{I}_2|+|\mathcal{G}_2|\right)+(1-\delta_1)\mathcal{D}. 
\]
Multiplying \eqref{finald} by $(1-\delta_1)$, and summing it with \eqref{sum-bg}, \eqref{sum-a} and \eqref{Y^2} with putting $C_*:=\frac{2\beta^2}{n_-\s_-}$, we find 
\begin{align*}
\begin{aligned}
&n_-\s_-\frac{\lambda^2}{\eps^3} \mathcal{R}_{\eps,\delta}(n)\\
&\quad\le-\frac{1}{C_*\delta_1}\left(\int_0^1W^2\,dy+2\int_0^1 W\,dy\right)^2+(1+C \delta_1)\int_0^1 W^2\,dy\\
&\qquad+\frac{2}{3}\int_0^1 W^3\,dy +C \delta_1\int_0^1 |W|^3\,dy  -(1-C \delta_1)\int_0^1 y(1-y)|\partial_y W|^2\,dy.
\end{aligned}
\end{align*}
Let $\delta^*$ be the constant in Lemma \ref{lem_poincare} corresponding to the constant $M$ of \eqref{controlW}. \\
Taking $\delta_1$ small enough such that $\max(C_*, C) \delta_1\le \delta^*$, therefore we have
\begin{align*}
\begin{aligned}
&n_-\s_-\frac{\lambda^2}{\eps^3} \mathcal{R}_{\eps,\delta}(n)\\
&\quad\le-\frac{1}{\delta_*}\left(\int_0^1W^2\,dy+2\int_0^1 W\,dy\right)^2+(1+ \delta_*)\int_0^1 W^2\,dy\\
&\qquad+\frac{2}{3}\int_0^1 W^3\,dy + \delta_*\int_0^1 |W|^3\,dy  -(1- \delta_*)\int_0^1 y(1-y)|\partial_y W|^2\,dy =: R_{\delta_*}(W).
\end{aligned}
\end{align*}
Then we have $R_{\delta_*}(W)\leq0$ by Lemma \ref{lem_poincare}. Therefore  $\mathcal{R}_{\eps,\delta}(n)\leq 0$.

\section{Proof of Proposition \ref{main_prop}}\la{proof_main_prop}
 

\subsection{Truncation of the big values of $|(n/\tiln)-1|$}\label{section-finale}
In order to use Proposition \ref{prop_near}, we need to show that  the values for $n$ such that $|(n/\tiln)-1|\geq\delta_1$ have a small effect. However, the value of $\delta_1$ is itself conditioned to the constant $K$ in Proposition \ref{prop_near}. Therefore, we need first to find a uniform bound on $\mathcal{Y}_g$ which is not yet conditioned on the level of truncation $\delta_1$.

We define a truncation on $|(n/\tiln)-1|$ with any constant $\theta\in(0,1/2)$ as follows: 
\be\la{def_bar_}
\bar{n}_\theta:=\begin{cases}
&n \mbox{ if } |\frac{n}{\tiln}-1|\leq \theta\\ 
&(1+\theta)\tiln \mbox{ if }  \frac{n}{\tiln}-1\geq \theta\\
&(1-\theta)\tiln \mbox{ if }  \frac{n}{\tiln}-1\leq-\theta.
\end{cases}
\ee 
Notice that 
\beq\label{st_bar}
\Big|\frac{\bar{n}_\theta}{\tiln}-1\Big|\leq \theta.
\eeq

\begin{lemma}\la{lem_ey}
There exist constants $\delta_0\in(0,1/2)$, $C, K>0$ such that 
for any $\eps, \lambda>0$ with $\delta_0^{-1}\eps<\lambda<\delta_0$, the following holds for $U\in\mathcal{H}$ whenever $|\mathcal{Y}(U)|\leq \vep^2$:
\be\la{eta_small_} \int_\bbr a' \Pi(n|\tiln)\,d\xi+ \intr a' |q-\tilq |^2\,d\xi \leq C\frac{\varepsilon^2}{\lambda},\ee and 
\be\la{y_g_small_} |\mathcal{Y}_g(\bar{n}_\theta)|\leq K \frac{\varepsilon^2}{\lambda}\quad
\mbox{ for any }  \theta\in(0,\frac{1}{2}) . 
\ee  
\end{lemma}

\begin{proof}
$\bullet$ {\it proof of \eqref{eta_small_}} : 
We consider $\delta_0$ small enough such that it is smaller than \eqref{delta_1}, and therefore there exists $C>0$ such that $ \s, \tiln \in (C^{-1},C)$.\\

First of all, using \eqref{der-eta} together with $a'=-\frac{\lambda}{\vep}\tiln'$ and ${\tilq}'=-\frac{{\tiln}'}{\sigma}$, we rewrite $\mathcal{Y}(U)$ in \eqref{badgood} as
\beq\label{re-Y}
\mathcal{Y}(U)= -{\intr}a'\Big(\frac{|q-\tilq|^2}{2}+{\Pi}(n|\tiln)\Big)d\xi-\frac{\vep}{\lambda}{\intr}aa'\Big(\frac{n-\tiln}{\tiln}-\frac{q-\tilq}{\s}\Big)d\xi.
\eeq
Then we have
\begin{align*}
\begin{aligned}
 &\intr a'\eta(U|\tilu)d\xi  \leq |\mathcal{Y}(U)|  +\frac{\vep}{\lambda}{\intr}aa' \Big|\frac{n-\tiln}{\tiln}-\frac{q-\tilq}{\s}\Big|d\xi \\
 &\leq \vep^2 +C\el\int_{ \{ |\frac{n}{\tiln}-1|\le\frac{1}{2} \}} a'|n-\tiln|d\xi +C\el\int_{ \{ |\frac{n}{\tiln}-1|>\frac{1}{2} \} } a'|n-\tiln|d\xi
 +C\el\intr a' |q-\tilq|d\xi.
\end{aligned}
\end{align*}
Thus we use \eqref{Q_loc_} and \eqref{upper_1/2} to have 
\begin{align*}
\begin{aligned}
 &\intr a'\eta(U|\tilu)d\xi \\
 &\leq \vep^2 +C\el\sqrt{\int_{
\{
|\frac{n}{\tiln}-1|\le\frac{1}{2}
\}
}a'|n-\tiln|^2d\xi}\cdot\sqrt{\intr a' d\xi}
+C\el\int_{ \{ |\frac{n}{\tiln}-1|>\frac{1}{2} \} } a'|n-\tiln|d\xi \\
&\quad\quad
 +C\el\sqrt{\intr a' |q-\tilq|^2d\xi}\cdot\sqrt{\intr a' d\xi}      \\
  &\leq \vep^2 +C\frac{\vep}{\sqrt{\lam}}\sqrt{\int_{
\{
|\frac{n}{\tiln}-1|\le\frac{1}{2}
\}
}a'
{\Pi}(n|\tiln)
d\xi}
+C\delta_0\int_{\{|\frac{n}{\tiln}-1|>\frac{1}{2}\}} a'
{\Pi}(n|\tiln)
d\xi \\
&\quad\quad+C\frac{\vep}{\sqrt{\lam}}\sqrt{\intr a' |q-\tilq|^2d\xi}  \\
&\le C\eel +\frac{1}{2}\intr a' \eta(U|\tilu)d\xi 
\end{aligned}
\end{align*} by taking $\delta_0$ small enough.
 Hence we have
\[
\intr a'\eta(U|\tilu)d\xi  \leq C\eel,
\]
which implies \eqref{eta_small_}.\\
 
 $\bullet$ {\it proof of \eqref{y_g_small_}} : 
 Let $\theta\in(0,1/2)$. 
Recall the functional $\mathcal{Y}_g$ and $\varphi$ in \eqref{badgood-n}. Since 
\[
|\varphi(\bar{n}_\theta)|\le C \Pi(\bar{n}_\theta |\tiln) + C |\bar{n}_\theta -\tiln|, 
\]
we have
\begin{align*}
\begin{aligned}
\mathcal{Y}_g(\bar{n}_\theta)&\le C {\intr}a' {\Pi}(\bar{n}_\theta|\tiln)^2 d\xi + C {\intr}a' |\bar{n}_\theta-\tiln|^2 d\xi + C\el {\intr}a' |\bar{n}_\theta-\tiln| d\xi + C{\intr}a' {\Pi}(\bar{n}_\theta|\tiln) d\xi.
\end{aligned}
\end{align*} 
Since it follows from \eqref{Q_loc_} with Remark \ref{rem_c_1_indep} and \eqref{st_bar} that
\[
{\Pi}(\bar{n}_\theta|\tiln) \le C_1|\bar{n}_\theta-\tiln|^2  \le C\Big|\frac{\bar{n}_\theta}{\tiln}-1\Big|\leq C,
\]
we have
\[
{\intr}a' {\Pi}(\bar{n}_\theta|\tiln)^2 d\xi \le C {\intr}a' {\Pi}(\bar{n}_\theta|\tiln) d\xi.
\]
Likewise, using \eqref{Q_loc_}, we have
\[
 {\intr}a' |\bar{n}_\theta-\tiln|^2 d\xi \le C {\intr}a' {\Pi}(\bar{n}_\theta|\tiln) d\xi,
\]
and 
\[
\el {\intr}a' |\bar{n}_\theta-\tiln| d\xi \le \frac{\eps}{\sqrt\lambda} \sqrt{ \intr a' |\bar{n}_\theta-\tiln|^2 d\xi} \le C\frac{\eps}{\sqrt\lambda} \sqrt{ \intr a' {\Pi}(\bar{n}_\theta|\tiln)  d\xi}.
\]
Since \eqref{Pi-mono} and \eqref{def_bar_} imply
\beq\label{Pi-com}
{\Pi}(\bar{n}_\theta|\tiln) \le \Pi (n|\tiln),
\eeq
we use \eqref{eta_small_} to find that there exists $K>0$ such that
\[
\mathcal{Y}_g(\bar{n}_\theta) \le C {\intr}a' {\Pi}(n|\tiln) d\xi +C\frac{\eps}{\sqrt\lambda} \sqrt{ \intr a' {\Pi}(n|\tiln)  d\xi}  \le K \frac{\varepsilon^2}{\lambda}.
\]
\end{proof}

From now until the end, we take and  fix  the constant $\delta_1$ from Proposition \ref{prop_near} associated to the constant $K$ of Lemma \ref{lem_ey}. 
In what follows, we use the simple notation: (without confusion)
$$
\bar n:=\bar n_{\delta_1},  \qquad \bar{U}:=(\bar n, q), \qquad \mathcal{B}:=\mathcal{B}_{\delta_1} \qquad\mbox{and}\qquad 
\mathcal{G}:=\mathcal{G}_{\delta_1}\qquad(\mbox{see }\,\eqref{bg-max}).
$$
Note that from Lemma \ref{lem_ey}, we have 
\begin{equation}\label{YC2}
|\mathcal{Y}_g(\bar n)|\leq K \frac{\eps^2}{\lambda}.
\end{equation}

In what follows, we will set $\Omega:=\{\xi~|~|\frac{n}{\tiln}-1| \leq\delta_1\}$. 

We decompose  $\mathcal{G}=\mathcal{G}_1^I+ \mathcal{G}_1^O+ \mathcal{G}_2+ \mathcal{D}$ where \begin{align}
\begin{aligned}\label{ggd}
&\mathcal{G}_1^I(U)=\frac{\sigma }{2}\int_\Omega a' \Big(q-\tilq +\varphi(n) \Big)^2 d\xi, \\
&\mathcal{G}_1^O(U)= \s\int_{\Omega^c} a' \frac{|q-\tilq|^2}{2}  d\xi,\\
&\mathcal{G}_2(U)=\s\intr a' {\Pi}(n | \tiln) d\xi ,\\
&\mathcal{D}(U)=  \intr a n\Big| \pa_\xi \Big(\log\frac{n}{\tiln}\Big)\Big|^2 d\xi.
\end{aligned}
\end{align}
 Notice that the functionals $\mathcal{G}_2, \mathcal{D}$ are as in \eqref{badgood-n} and they do not depend on $q$.

We first notice that it follows from \eqref{Pi-com} that
\begin{equation}\label{eq_G}
 \mathcal{G}_2(U)- \mathcal{G}_2(\bar U)=\sigma \int_\bbr a' \left( \Pi (n|\tiln)-{\Pi}(\bar{n} |\tiln) \right)\,d\xi\geq 0,
\end{equation}
which together with  \eqref{eta_small_} yields
\ben\label{l2}
0\leq \mathcal{G}_2(U)- \mathcal{G}_2(\bar U) \leq  C {\intr}a' {\Pi}(n|\tiln) d\xi \leq C\frac{\eps^2}{\lambda}.
\een

On the other hand, since $\bar{n}/\tiln$ is constant for any $n$ satisfying either $(n/\tiln)>1+\delta_1$ or $(n/\tiln)<1-\delta_1$ by the definition of $\bar n$, we see
\[
\mathcal{D}(\bar n)=\intr an\Big|\pa_\xi \log\frac{n}{\tiln} \Big|^2 {\mathbf 1}_{\{ |\frac{n}{\tiln}-1|\leq \delta_{1} \}} d\xi.
\]
Therefore we have
\be\la{P_g_diff}
   \mathcal{D}(n)-\mathcal{D}(\bar n)= {\intr}an\Big|\pa_\xi[\log\frac{n}{\tiln}]\Big|^2  {\mathbf 1}_{\{ |\frac{n}{\tiln}-1|> \delta_{1} \}} d\xi  \ge 0.
\ee  
Hence, since \eqref{YC2}, \eqref{eq_G} and \eqref{P_g_diff} together with \eqref{def_H_space} imply that for any $(n,q)\in \mathcal{H}$, $\bar n$ satisfies the assumptions \eqref{condn_for_main_prop}, Proposition \ref{prop_near} implies
\beq\label{near_ineq}
\mathcal{R}_{\eps,\delta_1}(\bar n)\le 0. 
\eeq

Before specifying the following proposition, we first recall \eqref{re-Y} as
\[
\mathcal{Y}(U)= -{\intr}a'\Big(\frac{|q-\tilq|^2}{2}+{\Pi}(n|\tiln)\Big)d\xi-\frac{\vep}{\lambda}{\intr}aa'\Big(\frac{n-\tiln}{\tiln}-\frac{q-\tilq}{\s}\Big)d\xi.
\]
We split $\mathcal{Y}$ into four parts $\mathcal{Y}_g$,  $\mathcal{Y}_b$, $\mathcal{Y}_l$, $\mathcal{Y}_s$ as follows: \\
\[
\mathcal{Y}= \mathcal{Y}_g +\mathcal{Y}_b +\mathcal{Y}_l+ \mathcal{Y}_s,
\]
where 
\begin{align*}
\begin{aligned}
\mathcal{Y}_g(U)&=-\int_\Omega a'\Big(\frac{|{{\varphi}}({n})|^2}{2}+{\Pi}(n|\tiln)\Big)d\xi -\frac{\vep}{\lambda} \int_\Omega aa'\Big(\frac{n-\tiln}{\tiln}+\frac{{{\varphi}}(n)}{\s}\Big)d\xi,\\
\mathcal{Y}_b(U)&= -\frac{1}{2}\int_\Omega a' \Big(q-\tilq +\varphi(n) \Big)^2 d\xi  + \int_\Omega a' \varphi(n) \Big(q-\tilq +\varphi(n) \Big) d\xi,\\
\mathcal{Y}_l(U)&=\frac{\vep}{\lambda}\frac{1}{\s} \int_\Omega aa'\Big(q-\tilq +\varphi(n) \Big)d\xi,\\
\mathcal{Y}_s(U)&=-\int_{\Omega^c}a'\Big(\frac{|q-\tilq|^2}{2}+{\Pi}(n|\tiln)\Big)d\xi-\frac{\vep}{\lambda}\int_{\Omega^c}aa'\Big(\frac{n-\tiln}{\tiln}-\frac{q-\tilq}{\s}\Big)d\xi.
\end{aligned}
\end{align*}

Notice that the functional $\mathcal{Y}_g$ is as in \eqref{badgood-n}. We also notice that $\mathcal{Y}_g$ consists of the terms related to $n$,  while $\mathcal{Y}_b$ and $\mathcal{Y}_l$ consist of   terms related to $q$. While $\mathcal{Y}_b$ is quadratic, and $\mathcal{Y}_l$ is linear in $q$. \\

For the bad terms $\mathcal{B}$ in \eqref{bg-max}, we will use the following notations :
\ben\label{bad0}
\mathcal{B}=\mathcal{B}_1+\mathcal{B}_2^I +\mathcal{B}_2^O +\mathcal{B}_3,
\een
where
\begin{align*}
\begin{aligned}
&\mathcal{B}_1(U):= -{\intr} a' \tilq {\Pi}(n|\tiln)d\xi  -\el{\intr}a''\frac{a}{\tiln} \Pi(n|\tiln) d\xi,\\
&\mathcal{B}_2^I(U):=\frac{\sigma}{2} \int_\Omega a' |\varphi(n)|^2 d\xi,\\
&\mathcal{B}_2^O(U) :=  -\int_{\Omega^c} a' \Big[{\Pi}(n|\tiln) + \Big( 1+ \el\frac{a}{\tiln}\Big) (n-\tiln)\Big](q-\tilq)  d\xi,\\
&\mathcal{B}_3(U) :=  -{\intr}a'\Big( 1+ \el\frac{a}{\tiln}\Big)n \Big(\log\frac{n}{\tiln}\Big)   \pa_\xi \Big(\log\frac{n}{\tiln}\Big) d\xi.  
\end{aligned}
\end{align*}
Notice that  $\mathcal{B}_1(U)=\mathcal{I}_1({n})$ and $\mathcal{B}_2^I(U)=\mathcal{B}_2^I(\bar{U})\le\mathcal{I}_2(\bar{n})$ in \eqref{badgood-n}.

We now state the following proposition.

\begin{proposition}\label{prop_out}
There exist constants $\delta_0\in(0,1/2), C, C^*>0$
  such that for any $\delta_0^{-1}\eps<\lambda<\delta_0$, the following statements hold.
\begin{itemize}
\item[1.] For any $U\in\mathcal{H}$ such that $|\mathcal{Y}(U)|\leq \eps^2$,
\begin{eqnarray*}
\label{n1}
&&|\mathcal{B}_1(U)-\mathcal{B}_1(\bar U)| \leq C\sqrt\frac{\eps}{\lambda}  \mathcal{D}(U),\\
\nonumber
&&|\mathcal{B}_2^O(U)|\le C\sqrt\frac{\eps}{\lambda}  \mathcal{D}(U),\\
\nonumber
&&|\mathcal{B}_3(U)| \le 
\delta_0^{1/3}
\mathcal{D}(U) + C\delta_0\frac{\eps}{\lambda} \mathcal{G}_2(\bar U),\\
\label{n2}
&&|\mathcal{B}(U)| \le C^*\frac{\eps^2}{\lambda} + 
\delta_0^{1/4}
 \mathcal{D}(U).
\end{eqnarray*}
\item[2.] For any $U\in\mathcal{H}$ such that $|\mathcal{Y}(U)|\leq \eps^2$ and $\mathcal{D}(U)\leq \frac{C^*}{4}\frac{\eps^2}{\lambda}$,
\begin{align}
\begin{aligned}\label{m1}
&|\mathcal{Y}_b(U)|^2+|\mathcal{Y}_l(U)|^2+|\mathcal{Y}_s(U)|^2 \\
&\quad\le  C\frac{\eps^2}{\lambda}\left(\sqrt{\frac{\eps}{\lambda}}\mathcal{D}(U)+\mathcal{G}_1^O(U) + \left(\frac{\lambda}{\eps}\right)^{1/4}\mathcal{G}_1^I(U) +\left(\frac{\eps}{\lambda}\right)^{1/4}\mathcal{G}_2(\bar U) \right).
\end{aligned}
\end{align}
\end{itemize}
\end{proposition}

\subsection{Proof of Proposition \ref{prop_out}}
We will first derive a point-wise estimate on $|n-\tiln|{\mathbf 1}_{\{ |\frac{n}{\tiln}-1|> \delta_{1} \}}$ as follows:

\begin{lemma}\la{lem_pointwise}
For a sufficiently small $\delta_0>0$, there exists $C>0$  such that for any $\delta_0^{-1}\eps<\lambda<\delta_0$ and any $U\in\mathcal{H}$ satisfying $|\mathcal{Y}(U)|\leq \vep^2$, the following estimates hold:
    \be\la{est_pointwise}\begin{split}
    &|n(\xi)-\tiln(\xi)|\leq C\Big(\frac{1}{\vep}+|\xi|\Big)\mathcal{D}(n)
   \end{split}\ee
whenever $\xi\in\mathbb{R}$ satisfies \be\la{assump1}\Big |\frac{n(\xi)}{\tiln(\xi)}-1\Big|\geq \delta_1.\ee
   \end{lemma}
   \begin{remark}
   Recall that we assumed $\tiln(0)=(n_-+n_+)/2$.
   \end{remark}
   \begin{proof}
We set $\alpha  :=\frac{1}{\lam}  \int_{-1/\vep}^{1/\vep}a'\,d\xi$.  Using $\frac{1}{\lam}  \intr a' d\xi =1$ and $a'=(\lambda/\eps)|\tiln'|$ together with \eqref{tilde_estimate}, we obtain
$$ \frac{1}{2}(1-e^{-1/\s_-})\leq \alpha \leq 1.$$ 
Notice that $ \frac{1}{2}(1-e^{-1/\s_-})$ is a positive constant.\\
Since \eqref{eta_small_} implies
$$\int_{-1/\vep}^{1/\vep} a'{\Pi}(n|\tiln)d\xi\leq C\eel,$$
we have
  $$\int_{-1/\vep}^{1/\vep}\frac{a'}{\lam\alpha}{\Pi}(n|\tiln)d\xi\leq C\Big(\frac{\vep}{\lam}\Big)^2.$$ 
Since $\int_{-1/\vep}^{1/\vep}\frac{a'}{\lam\alpha}\,d\xi=1$,  there exists a point $\xi_0\in[-\frac{1}{\vep},\frac{1}{\vep}]$ such that
   \be\la{est_pi_temp}   {\Pi}(n(\xi_0)|\tiln(\xi_0))\leq   \tilde{C}\Big(\frac{\vep}{\lam}\Big)^2\le \tilde{C}\cdot(\delta_0)^2\ee where $\tilde{C}$ is some constant. We take $\delta_0$ small enough to get
$$\tilde{C}\cdot(\delta_0)^2\leq C_2/2$$ where $C_2$ is the constant in \eqref{Q_global_} by plugging $\delta=\delta_1$. 
  
We observe that \eqref{Q_loc_} and \eqref{Q_global_} imply
\begin{align*}
\begin{aligned}
 {\Pi}(n(\xi_0)| \tiln(\xi_0)) &\ge \min\Big(C_1^{-1}|n(\xi_0)-\tiln(\xi_0)|^2, C_2^{-1} (1+n(\xi_0)\log^+ \frac{n(\xi_0)}{\tiln(\xi_0)})\Big)\\
&\ge \min\Big(C_1^{-1}|n(\xi_0)-\tiln(\xi_0)|^2, C_2^{-1} \Big).
\end{aligned}
\end{align*}  Then from \eqref{est_pi_temp}, we get
\begin{align*}
\begin{aligned}
 |n(\xi_0)-\tiln(\xi_0)|^2  &\le C_1{\Pi}(n(\xi_0)| \tiln(\xi_0))
 \le C_1\tilde{C}\cdot(\delta_0)^2.
\end{aligned}
\end{align*} 
Thus, by taking  $\delta_0$ small enough, we can assume that
\be\la{xi_est}
\Big |\frac{n(\xi_0)}{\tiln(\xi_0)}-1\Big| 
\leq \min\{\frac{\delta_1}{2},\frac{(\sqrt{1+\delta_1}-1)(\sqrt{1-(\delta_1/2)}+1)}{2}\}. 
\ee
For the reference point $\xi_0$, since for any $\xi\in\mathbb{R}$,
\[
\begin{split}
\sqrt{\frac{n}{\tiln}(\xi)}-\sqrt{\frac{n}{\tiln}(\xi_0)}& 
=\int_{\xi_0}^{\xi}\frac{d}{d\zeta}\sqrt{\frac{n}{\tiln}(\zeta)}d\zeta\\&=
\int_{\xi_0}^{\xi}\frac{1}{2}\sqrt{\frac{\tiln}{n}(\zeta)}\frac{d}{d\zeta}{\frac{n}{\tiln}(\zeta)}d\zeta=\int_{\xi_0}^{\xi}\frac{1}{2}\sqrt{\frac{n}{\tiln}(\zeta)}\frac{d}{d\zeta}\log\Big({\frac{n}{\tiln}(\zeta)}\Big)d\zeta,
\end{split}
\]
we have
 \be\la{cal_differentiation_2}\begin{split}
&\Big|\sqrt{\frac{n}{\tiln}(\xi)}-\sqrt{\frac{n}{\tiln}(\xi_0)}\Big| 
=\Big| 
\int_{\xi_0}^{\xi}\frac{1}{2\sqrt{a(\zeta)n(\zeta)}}\sqrt{\frac{n}{\tiln}(\zeta)}\sqrt{a(\zeta)n(\zeta)}\frac{d}{d\zeta}\log\Big({\frac{n}{\tiln}(\zeta)}\Big)d\zeta
\Big|\\
&\quad\quad\quad\leq \sqrt{
\int_{\xi_0}^{\xi}\frac{1}{4{a(\zeta)n(\zeta)}}{\frac{n}{\tiln}(\zeta)}
d\xi}
\sqrt{\int_{\xi_0}^{\xi} {a(\zeta)n(\zeta)}\Big|\frac{d}{d\zeta}\log\Big({\frac{n}{\tiln}(\zeta)}\Big)\Big|^2d\zeta}\\
&\quad\quad\leq \sqrt{\frac{1}{2 n_-}
\int_{\xi_0}^{\xi}1
d\xi}
\sqrt{\mathcal{D}(n)}\leq \sqrt{\frac{1}{2 n_-}}
\sqrt{|\xi-\xi_0|}
\sqrt{\mathcal{D}(n)}
\leq\sqrt{\frac{1}{2 n_-}}
\sqrt{|\xi|+\frac{1}{\vep}}
\sqrt{\mathcal{D}(n)}.
\end{split}\ee
 
On the other hand, we claim that there exists $L=L(\delta_1)>0$ such that if $y>0$ and $y_0>0$ with $$|y_0-1|\leq \min\{\frac{\delta_1}{2},\frac{(\sqrt{1+\delta_1}-1)(\sqrt{1-(\delta_1/2)}+1)}{2}\}\quad\mbox{and } |y-1|\geq\delta_1
,$$ then 
\beq\label{claim-yy}
|y-1|\leq L|\sqrt{y}-\sqrt{y_0}|^2.
\eeq
Indeed, we can split it into two cases: $0<y\leq 1-\delta_1$ and $y\geq1+\delta_1$. \\
Denote $\beta:=|\sqrt{y}-\sqrt{y_0}|>0$.\\
For the first case $0<y\leq1-\delta_1$, since
$y\leq 1-\delta_1<1-(\delta_1/2)\leq y_0\leq 1+(\delta_1/2)$, we have
$$\frac{\delta_1}{2}\leq|y-y_0|\leq \beta|\sqrt{y}+\sqrt{y_0}|\leq
2\beta\sqrt{y_0}\leq 2\beta\sqrt{1+(\delta_1/2)}\leq 4\beta.$$ 
Thus we get
$1\leq\frac{8\beta}{\delta_1}$. Therefore
$$|y-1|=1-y\leq 1=1^2\leq\frac{64}{(\delta_1)^2}\beta^2.$$

For the second case $y\geq 1+\delta_1$, since
$$|\sqrt{y_0}-1|=\frac{|y_0-1|}{\sqrt{y_0}+1}\leq\frac{|y_0-1|}{\sqrt{1-(\delta_1/2)}+1}\leq  \frac{(\sqrt{1+\delta_1}-1)}{2}\leq\frac{ \sqrt{y}-1}{2}.
$$
we have
\beq\la{instant-y}
\beta=|(\sqrt{y}-1)-(\sqrt{y_0}-1)|
\geq|\sqrt{y}-1|-|\sqrt{y_0}-1|\geq \frac{ \sqrt{y}-1}{2}.
\eeq
Thus we get 
 $$1+\delta_1\leq y\leq(2\beta+1)^2,$$
which yields 
\[
0<\delta_1\leq 4\beta(\beta+1).
\]
Let $\beta_0=\beta_0(\delta_1)$ be the positive constant satisfying $4\beta_0(\beta_0+1)=\delta_1$. \\
Since $4\beta_0(\beta_0+1) \le 4\beta(\beta+1),$ we have $\beta\geq\beta_0$ so $1\leq \frac{\beta}{\beta_0}$.\\
Therefore, using \eqref{instant-y}, we get
$$|y-1|=y-1=(\sqrt{y}-1)((\sqrt{y}-1)+2)\leq 2\beta(2\beta+2)
\leq4\beta(\beta+\frac{\beta}{\beta_0})=4(1+\frac{1}{\beta_0})\beta^2.
$$ It proves the above claim \eqref{claim-yy} by taking $L:=\frac{64}{(\delta_1)^2}+4(1+\frac{1}{\beta_0})$.\\

By considering $y:=\frac{n}{\tiln}(\xi)$ and $y_0:=\frac{n}{\tiln}(\xi_0)$ in the claim \eqref{claim-yy} together with \eqref{xi_est} and \eqref{assump1}, it follows from \eqref{cal_differentiation_2} that
 \ben\la{est_pointwise_}\begin{split}
    &|n(\xi)-\tiln(\xi)|\leq 
(n_-)|y-1|\leq(n_-)L|\sqrt{y}-\sqrt{y_0}|^2    
    \leq C\Big(\frac{1}{\vep}+|\xi|\Big)\mathcal{D}(n).
   \end{split}\een
\end{proof}

\begin{lemma}\la{lem_B_log}
 Under the same assumption as in Lemma \ref{lem_pointwise}, we have
\be\la{est_B1}
\int_{\Omega^c} a' \Big(1+n\log^+ \frac{n}{\tiln}\Big) d\xi\leq C\sqrt{\el}\mathcal{D}(n),
\ee
\be\la{est_B2}
\int_{\Omega^c} a' \Big(1+n\Big[\log^+ \frac{n}{\tiln}\Big]^2\Big)d\xi\leq C\sqrt{\el}\mathcal{D}(n),
\ee
\be\la{est_B3}
\int_{\Omega^c} a' |q-\tilq| \Big(1+n\log^+ \frac{n}{\tiln}\Big) d\xi \leq C\sqrt{\el}\mathcal{D}(n).
\ee
\end{lemma}

\begin{proof}
 $\bullet$ {\it proof of \eqref{est_B1}} : 
Since $\log^+ \frac{n}{\tiln}\leq \frac{1}{\log(1+\delta_1)}\Big[\log^+ \frac{n}{\tiln}\Big]^2$ whenever $|\frac{n}{\tiln}-1|\geq\delta_1$, the desired result \eqref{est_B1} follows from \eqref{est_B2}.

 $\bullet$ {\it proof of \eqref{est_B2}} :   
Since if $n$ satisfies $\frac{n}{\tiln}-1\le -\delta_1$ then 
\[
\log^+ \frac{n}{\tiln}=0,
\]
and
\[
|n-\tiln|=\tiln-n\geq (\delta_1\tiln)\geq\Big(\delta_1\frac{n_-}{2}\Big)>0,
\quad {\Pi}(n|\tiln)
\geq C_{2}>0,
\]
we find that there exists a constant $C>0$ (depending on $\delta_1$) such that
   \ben\la{log_1}\begin{split}
&   \Big(1+n\Big[\log^+ \frac{n}{\tiln}\Big]^2\Big){1}_{\{\frac{n}{\tiln}-1\leq-\delta_1\}} \leq C \sqrt{{\Pi}(n|\tiln)} |n-\tiln| {1}_{\{\frac{n}{\tiln}-1\leq-\delta_1\}}.
   \end{split}\een
   
Since if $n$ satisfies $\frac{n}{\tiln}-1\ge\delta_1$ then (by \eqref{Q_global_})
\ben\la{m_est_2}
 {\Pi}(n|\tiln)\geq \frac{1}{C_{2}}(1+[\log(1+\delta_1)]\cdot n ),
\een 
using the inequality:   
 \be\la{log_com}    
\Big(1+n\Big[\log^+ \frac{n}{\tiln}\Big]^2\Big)\leq 
1+n\tau^2 \Big( \frac{n}{\tiln}\Big)^{1/6}  
\le
1+[\tau^2(2/n_-)^{1/6}]n^{7/6},\quad \tau:=\sup_{y\in[1+\delta_1,\infty)}\frac{\log y}{y^{1/12}}<\infty,
\ee
we find that there exists a constant $C>0$ such that 
\[  
 \Big(1+n\Big[\log^+ \frac{n}{\tiln}\Big]^2\Big){1}_{\{\frac{n}{\tiln}-1\geq\delta_1\}}\leq  \sqrt{{\Pi}(n|\tiln)} |n-\tiln| {1}_{\{\frac{n}{\tiln}-1 \ge\delta_1\}}.  
\]
Indeed for large $n$, the left-hand side is bounded above by $C(1+Cn^{7/6})$ while the right one is bounded below by $ \frac{1}{C}(1+\frac{1}{C}n^{3/2})$.\\
   
 By combining these two cases, we obtain 
\[
 \Big(1+n\Big[\log^+ \frac{n}{\tiln}\Big]^2\Big){1}_{\{|\frac{n}{\tiln}-1|\geq\delta_1\}}\leq  \sqrt{{\Pi}(n|\tiln)} |n-\tiln| {1}_{\{|\frac{n}{\tiln}-1| \ge\delta_1\}}.  
\] 
Then, we have
  \ben\begin{split}
&\intr a' \Big(1+n\Big[\log^+ \frac{n}{\tiln}\Big]^2\Big){1}_{\{|\frac{n}{\tiln}-1|\geq\delta_1\}} d\xi \le \intr a' \sqrt{\Pi(n|\tiln)} |n-\tiln| {1}_{\{|\frac{n}{\tiln}-1|\geq\delta_1\}} d\xi \\
&\quad \le  \underbrace{\int_{|\xi|\le\frac{1}{\eps}\sqrt{\frac{\lambda}{\eps}}} a'\sqrt{\Pi(n|\tiln)} |n-\tiln| {1}_{\{|\frac{n}{\tiln}-1|\geq\delta_1\}} d\xi}_{=:J_1} +\underbrace{\int_{|\xi|\ge\frac{1}{\eps}\sqrt{\frac{\lambda}{\eps}}} a'\sqrt{\Pi(n|\tiln)} |n-\tiln| {1}_{\{|\frac{n}{\tiln}-1|\geq\delta_1\}} d\xi}_{=:J_2}.
\end{split}\een
Since it follows from \eqref{upper_1/2} that $\delta_1\le C|n-\tiln| \le C\Pi(n|\tiln)$ whenever $|\frac{n}{\tiln}-1|\geq\delta_{1}$, we use \eqref{est_pointwise} and \eqref{eta_small_} to find that there exists a constant $C>0$ (depending on $\delta_1$) such that
\ben\begin{split}
J_1&\leq\left(\sup_{\left[-\sqrt{\frac{\lambda}{\eps^3}},\sqrt{\frac{\lambda}{\eps^3}}\right]\cap\Omega^c}|{n}-{\tiln}| \right)  \int_{|\xi|\le\frac{1}{\eps}\sqrt{\frac{\lambda}{\eps}}} a' \sqrt{\Pi(n|\tiln)} d\xi \\
& \leq C\frac{1}{\eps}\sqrt{\frac{\lambda}{\eps}} \mathcal{D}(U)   \int_\bbr a' \Pi(n|\tiln) \,d\xi\\
& \leq C\sqrt{\frac{\eps}{\lambda}} \mathcal{D}(U).
\end{split}\een
Using  \eqref{est_pointwise} and \eqref{eta_small_}, we have
\ben\begin{split}
J_2&\leq C\mathcal{D}(U)  \int_{|\xi|\geq\frac{1}{\eps}\sqrt{\frac{\lambda}{\eps}}} a' \sqrt{\Pi(n|\tiln)} \left(|\xi|+\frac{1}{\eps}\right)\,d\xi\\
& \leq  C\mathcal{D}(U) \Big( \intr a' \Pi(n|\tiln) \,d\xi\Big)^{1/2} \Big( \int_{|\xi|\geq\frac{1}{\eps}\sqrt{\frac{\lambda}{\eps}}} a' |\xi|^2 \,d\xi\Big)^{1/2}\\
& \leq  C \mathcal{D}(U) \sqrt{\frac{\eps^2}{\lambda}} \Big( \int_{|\xi|\geq\frac{1}{\eps}\sqrt{\frac{\lambda}{\eps}}} a' |\xi|^2 \,d\xi\Big)^{1/2}.
\end{split}\een
Notice that
\[
 \int_{|\xi|\geq\frac{1}{\eps}\sqrt{\frac{\lambda}{\eps}}} a' |\xi|^2 \,d\xi \le C\eps\lambda\int_{|\xi|\geq\frac{1}{\eps}\sqrt{\frac{\lambda}{\eps}}}e^{-c\eps|\xi|} |\xi|^2 \,d\xi \le C\frac{\lambda}{\eps^2} \int_{|\xi|\geq \sqrt{\frac{\lambda}{\eps}}} |\xi|^2 e^{-c|\xi|}d\xi.
\]
Taking $\delta_0$ small enough such that for any $\eps/\lambda\leq \delta_0$, $|\xi|^2\leq e^{(c/2)|\xi|}$ for $\xi\geq \sqrt{\lambda/\eps}$ and
$$
\int_{|\xi|\geq \sqrt{\frac{\lambda}{\eps}}} |\xi|^2 e^{-c|\xi|}d\xi \le \int_{|\xi|\geq \sqrt{\frac{\lambda}{\eps}}} e^{-\frac{c}{2}|\xi|}d\xi = Ce^{-\frac{c}{2}\sqrt{\frac{\lambda}{\eps}}}\leq C\frac{\eps}{\lambda},
$$  we have
\[
J_2\le C\sqrt{\frac{\eps}{\lambda}} \mathcal{D}(U),
\]
which gives the desired estimate.

$\bullet$ {\it proof of \eqref{est_B3}} :   
Following the same estimates together with \eqref{log_com} as above, and using $\log^+ \frac{n}{\tiln}\leq \frac{1}{\log(1+\delta_1)}\Big[\log^+ \frac{n}{\tiln}\Big]^2$, we find that there exists a constant $C>0$ such that 
\[  
 \Big(1+n \log^+ \frac{n}{\tiln} \Big){1}_{\{\frac{n}{\tiln}-1\geq\delta_1\}}\leq  {\Pi}(n|\tiln)^{1/4} |n-\tiln| {1}_{\{\frac{n}{\tiln}-1 \ge\delta_1\}}.  
\]
Indeed for large $n$, the right-hand side is bounded below by $ \frac{1}{C}(1+\frac{1}{C}n^{5/4})$.\\

Then, we have
  \ben\begin{split}
&\intr a' |q-\tilq| \Big(1+n \log^+ \frac{n}{\tiln}\Big){1}_{\{|\frac{n}{\tiln}-1|\geq\delta_1\}} d\xi\\
&\quad \le  \underbrace{\int_{|\xi|\le\frac{1}{\eps}\sqrt{\frac{\lambda}{\eps}}} a'  |q-\tilq|{\Pi}(n|\tiln)^{1/4} |n-\tiln| {1}_{\{|\frac{n}{\tiln}-1|\geq\delta_1\}} d\xi}_{=:K_1} \\
&\qquad+\underbrace{\int_{|\xi|\ge\frac{1}{\eps}\sqrt{\frac{\lambda}{\eps}}} a'|q-\tilq|{\Pi}(n|\tiln)^{1/4}|n-\tiln| {1}_{\{|\frac{n}{\tiln}-1|\geq\delta_1\}} d\xi}_{=:K_2}.
\end{split}\een
Using the same argument as in $J_1$ above, we have
\ben\begin{split}
K_1&\leq C\frac{1}{\eps}\sqrt{\frac{\lambda}{\eps}} \mathcal{D}(U)   \int_\bbr a' |q-\tilq|{\Pi}(n|\tiln)^{1/2} \,d\xi\\
& \leq C\frac{1}{\eps}\sqrt{\frac{\lambda}{\eps}} \mathcal{D}(U)  \Big( \intr a' |q-\tilq|^2 \,d\xi\Big)^{1/2} \Big( \intr a' \Pi(n|\tiln) \,d\xi\Big)^{1/2}\\
& \leq C\sqrt{\frac{\eps}{\lambda}} \mathcal{D}(U).
\end{split}\een
Since $|q-\tilq|{\Pi}(n|\tiln)^{1/4} \le C\eta(U|\tilde U)^{3/4}$, we have
\ben\begin{split}
K_2&\leq C\mathcal{D}(U)  \int_{|\xi|\geq\frac{1}{\eps}\sqrt{\frac{\lambda}{\eps}}} a' \eta(U|\tilde U)^{3/4} \left(|\xi|+\frac{1}{\eps}\right)\,d\xi\\
& \leq  C\mathcal{D}(U) \Big( \intr a'  \eta(U|\tilde U) \,d\xi\Big)^{3/4} \Big( \int_{|\xi|\geq\frac{1}{\eps}\sqrt{\frac{\lambda}{\eps}}} a' |\xi|^4 \,d\xi\Big)^{1/4}\\
& \leq  C \mathcal{D}(U) \Big( \frac{\eps^2}{\lambda}\Big)^{3/4} \Big( \frac{\lambda}{\eps^4} \Big)^{1/4} = C\sqrt{\frac{\eps}{\lambda}} \mathcal{D}(U).
\end{split}\een
 \end{proof}
   
\subsubsection{\bf Proof of \eqref{n1}}
We first use \eqref{Q_global_} and \eqref{est_B1} to have
\be\label{B1est}
\begin{split}
|\mathcal{B}_1(U)-\mathcal{B}_1(\bar U)|  &\le C \intr a' \Big|\Pi(n|\tiln)-\Pi(\bar n|\tiln) \Big| d\xi \\
&\le  C \intr a'  \Pi(n|\tiln){1}_{\{|\frac{n}{\tiln}-1|\geq\delta_1\}} d\xi\\
&\le  C \intr a' \Big(1+n\log^+ \frac{n}{\tiln}\Big) {1}_{\{|\frac{n}{\tiln}-1|\geq\delta_1\}} d\xi\\
&\leq C\sqrt{\el}\mathcal{D}(n).
\end{split}\ee

We use \eqref{upper_1/2}, \eqref{Q_global_} and \eqref{est_B3} to have
\ben\begin{split}
|\mathcal{B}_2^O (U)|&\le \int_{\Omega^c} a' {\Pi}(n|\tiln) |q-\tilq| d\xi + C\int_{\Omega^c} a' |n-\tiln| |q-\tilq| d\xi \\
&\le C\int_{\Omega^c} a' {\Pi}(n|\tiln) |q-\tilq| d\xi\\
&\le C \int_{\Omega^c} a' \Big(1+n\log^+ \frac{n}{\tiln}\Big) |q-\tilq| d\xi \leq C\sqrt{\el}\mathcal{D}(n).
\end{split}\een

We use Young's inequality to have
\ben\begin{split}
|\mathcal{B}_3(U)| &\le C {\intr}a' \sqrt n \Big(\log\frac{n}{\tiln}\Big) \sqrt n   \pa_\xi \Big(\log\frac{n}{\tiln}\Big) d\xi\\
 &\le \delta_0 \mathcal{D}(U) + \frac{C}{\delta_0} {\intr}|a'|^2  n \Big(\log\frac{n}{\tiln}\Big)^2 d\xi\\
 &\le \delta_0 \mathcal{D}(U) + \underbrace{\frac{C\eps\lambda}{\delta_0} {\intr} a'  n \Big(\log\frac{n}{\tiln}\Big)^2 d\xi}_{=:B_4(n)}.
 \end{split}\een  
We separate the remaining term $B_4(n)$ into
\[
|B_4(n)|\le |B_4(n)-B_4(\bar n)|+ |B_4(\bar n)|.
\] 
Since there exists a constant $C>0$ such that 
\[
n \Big(\log\frac{n}{\tiln}\Big)^2 {1}_{\{|\frac{n}{\tiln}-1|\geq\delta_1\}} \le C \Big[1+n\Big(\log^+ \frac{n}{\tiln}\Big)^2 \Big] {1}_{\{|\frac{n}{\tiln}-1|\geq\delta_1\}},
\]
we use \eqref{est_B2} to have
\[
 |B_4(n)-B_4(\bar n)| \le C\eps {\intr} a'   \Big[1+n\Big(\log^+ \frac{n}{\tiln}\Big)^2 \Big] {1}_{\{|\frac{n}{\tiln}-1|\geq\delta_1\}} \le C\sqrt{\el}\mathcal{D}(n).
\]
Since there exists a constant $C>0$ such that 
$$|\log y|\leq C|y-1|\quad \mbox{for any $y$ satisfying $|y-1|\leq \delta_1$},$$ 
using $ \bar n\leq (1+\delta_1)\tiln \leq C$ and \eqref{Q_loc_}, we have
\[
|B_4(\bar n)| \le \frac{C\eps\lambda}{\delta_0} {\intr} a'  \Big(\frac{\bar n}{\tiln}-1\Big)^2 d\xi \le  \frac{C\eps\lambda}{\delta_0} {\intr} a'  \Pi(\bar n|\tiln) d\xi.
\]
Using $\eps<\delta_0(\eps/\lambda)$, we have
\[
|B_4(\bar n)| \le C\delta_0  \frac{\eps}{\lambda} \mathcal{G}_2(\bar U).
\]
Therefore, by taking $\delta_0$ small enough, we get
\[
|\mathcal{B}_3 (U)| \le
\delta_0^{1/3} 
 \mathcal{D}(n) + C\delta_0  \frac{\eps}{\lambda} \mathcal{G}_2(\bar U).
\]

\subsubsection{\bf Proof of \eqref{n2}}
Using \eqref{eq_G} and \eqref{eta_small_}, we have
\[
|\mathcal{B}_1(\bar U)|\le  C{\intr} a'  \Pi(\bar n|\tiln) d\xi\le C  {\intr} a'  \Pi(n|\tiln) d\xi\le C\frac{\eps^2}{\lambda}.
\]
Since $|\bar n| \le C$, using \eqref{Q_loc_}, \eqref{eq_G} and \eqref{eta_small_}, we have
\begin{align}
\begin{aligned}\label{B2est}
|\mathcal{B}_2^I(U)|&\le C\int_\Omega a' \Big(\Pi(\bar n|\tiln) +\left(1+\frac{\eps}{\lambda} \frac{a}{\tiln} \right)(\bar n-\tiln) \Big)^2 d\xi \\
&\le C\int_\Omega a' \Big(\Pi(\bar n|\tiln)^2 + |\bar n-\tiln|^2 \Big) d\xi\\
 &\le C\int_\Omega a' \Pi(\bar n|\tiln) d\xi  \le C\frac{\eps^2}{\lambda}.
\end{aligned}
\end{align}
Hence, combining these estimates with \eqref{n1}, using \eqref{eta_small_}, and taking $\delta_0$ small enough, there exists $C^*>0$ such that
\[
|\mathcal{B}(U)| \le C^* \frac{\eps^2}{\lambda}+
\delta_0^{1/4} 
  \mathcal{D}(U).
\]

\subsubsection{\bf Proof of \eqref{m1}} We split the proof in two steps.
\vskip0.2cm
\noindent{\it Step 1:} 
We use the good term $\mathcal{G}_1^{I}$ defined in \eqref{ggd} and \eqref{B2est} to have
\begin{align*}
\begin{aligned}
|\mathcal{Y}_b(U)| &\le C \mathcal{G}_1^{I}(U) + C \int_\Omega a' |\varphi(n)|^2 d\xi \\
&\le C \mathcal{G}_1^{I}(U) + C |\mathcal{B}_2^I(U)|\\
&\le C\Big( \mathcal{G}_1^{I}(U) + \mathcal{G}_2(\bar U) \Big).
\end{aligned}
\end{align*}
In particular, since
\[
\mathcal{G}_1^{I}(U) \le C \int_{\Omega} a' \Big(\frac{|q-\tilq|^2}{2} + |\varphi(n)|^2\Big)  d\xi \le C \int_{\Omega} a' \frac{|q-\tilq|^2}{2}  d\xi +C  \mathcal{G}_2(\bar U),
\]
we use \eqref{eta_small_} to have
\beq\label{Ybb}
|\mathcal{Y}_b(U)| \le C \frac{\eps^2}{\lambda}.
\eeq

We use the notations $Y_1^s, Y_2^s, Y_3^s$ and $Y_4^s$ for the terms of $\mathcal{Y}_s$ as follows:
\[
\mathcal{Y}_s =
\underbrace{-\int_{\Omega^c} a' {\Pi}(n|\tiln) d\xi}_{=:Y_1^s} \underbrace{-\frac{\vep}{\lambda}\int_{\Omega^c}aa' \frac{n-\tiln}{\tiln} d\xi}_{=:Y_2^s} \underbrace{-\int_{\Omega^c} a' \frac{|q-\tilq|^2}{2} d\xi}_{=:Y_3^s} \underbrace{+\frac{\vep}{\lambda}\int_{\Omega^c}aa' \frac{q-\tilq}{\s} d\xi}_{=:Y_4^s}.
\]
Using \eqref{B1est}, we have
\[
|Y_1^s(U)| = \intr a'  \Pi(n|\tiln){1}_{\{|\frac{n}{\tiln}-1|\geq\delta_1\}} d\xi \leq C\sqrt{\el}\mathcal{D}(n).
\]
Using \eqref{upper_1/2}, we have
\[
|Y_2^s(U)| \le C\frac{\vep}{\lambda} \int_{\Omega^c} a' |n-\tiln| d\xi \le C\intr a'  \Pi(n|\tiln){1}_{\{|\frac{n}{\tiln}-1|\geq\delta_1\}} d\xi \leq C\sqrt{\el}\mathcal{D}(n).
\]
We use $\mathcal{G}_1^{O}$ defined in \eqref{ggd} to control
\[
|Y_3^s(U)|  \le C\mathcal{G}_1^O(U).
\]
Therefore, we have
\beq\label{Ys123}
|Y_1^s(U)|+|Y_2^s(U)|+|Y_3^s(U)| \le C\sqrt{\el}\mathcal{D}(n) + C\mathcal{G}_1^O(U).
\eeq
Using \eqref{eta_small_} together with the assumption $|\mathcal{D}(U)|\leq C\frac{\eps^2}{\lambda}$, we have
\beq\label{Ysb}
|Y_1^s(U)|+|Y_2^s(U)|+|Y_3^s(U)|  \le C \frac{\eps^2}{\lambda}.
\eeq

\noindent{\it Step 2:}
First of all, using Young's inequality and \eqref{B2est}, we estimate
\begin{align*}
\begin{aligned}
\left|  \int_\Omega a' \Big(q-\tilq +\varphi(n) \Big) \varphi(n) d\xi \right|&\le\left(\frac{\lambda}{\eps}\right)^{1/4}\mathcal{G}_1^I(U)+C\left(\frac{\eps}{\lambda}\right)^{1/4}\int_\Omega a' |\varphi(n)|^2\,d\xi\\
& \le\left(\frac{\lambda}{\eps}\right)^{1/4}\mathcal{G}_1^I(U)+C\left(\frac{\eps}{\lambda}\right)^{1/4} \mathcal{G}_2(\bar U).
\end{aligned}
\end{align*}
Then we have
\[
|\mathcal{Y}_b(U)|\le C\left(\frac{\lambda}{\eps}\right)^{1/4}\mathcal{G}_1^I(U)+C\left(\frac{\eps}{\lambda}\right)^{1/4} \mathcal{G}_2(\bar U).
\]
Therefore, this together with \eqref{Ys123}, \eqref{Ybb} and \eqref{Ysb} implies
\begin{align*}
\begin{aligned}
&|\mathcal{Y}_b(U)|^2 + |Y_1^s(U)|^2+|Y_2^s(U)|^2+|Y_3^s(U)|^2\\
&\quad\le C\frac{\eps^2}{\lambda}\left(\sqrt{\frac{\eps}{\lambda}}\mathcal{D}(U)+\mathcal{G}_1^O(U) + \left(\frac{\lambda}{\eps}\right)^{1/4}\mathcal{G}_1^I(U) +\left(\frac{\eps}{\lambda}\right)^{1/4}\mathcal{G}_2(\bar U) \right).
\end{aligned}
\end{align*}
We use H\"older's inequality to have
\begin{align*}
\begin{aligned}
&|Y_4^s(U)|^2\leq C\left(\frac{\eps}{\lambda}\right)^2 \left(\int_\bbr|a'|\,d\xi\right)\int_{\Omega^c} a' |q-\tilq|^2\,d\xi \leq C\frac{\eps^2}{\lambda} \mathcal{G}_1^O(U),\\
&|Y_l(U)|^2\leq C\left(\frac{\eps}{\lambda}\right)^2 \left(\int_\bbr|a'|\,d\xi\right)\int_{\Omega} a' \left(q-\tilq +\varphi(n) \right)^2\,d\xi \leq C\frac{\eps^2}{\lambda} \mathcal{G}_1^I(U).
\end{aligned}
\end{align*}
Hence we have \eqref{m1}

\subsection{Conclusion}
We are now ready to complete the proof of Proposition \ref{main_prop}. We split the proof into two steps, depending on the strength of the dissipation term $\mathcal{D}(U)$. 
\vskip0.2cm
\noindent{\it Step 1:}  
We first consider the case of
$
\mathcal{D}(U)\geq 4 C^* \frac{\eps^2}{\lambda}, $  where the constant $C^*$ is defined as in Proposition \ref{prop_out}.
Then using $\eqref{n2}$ and taking $\delta_0$ small enough, we have
\begin{align*}
\begin{aligned}
\mathcal{R}(U)&:= -\frac{1}{\varepsilon^4}|\mathcal{Y}(U)|^2 + \mathcal{B}(U)+\delta_0\frac{\eps}{\lambda} |\mathcal{B}(U)|-\mathcal{G}_1^I(U)- \mathcal{G}_1^O(U)- \mathcal{G}_2(U)
-(1-\delta_0)\mathcal{D}(U)\\
&\leq 2|\mathcal{B}(U)|-(1-\delta_0) \mathcal{D}(U)\\
& \leq 2C^*\frac{\eps^2}{\lambda}-\left(1-\delta_0
-2\delta_0^{1/4}
\right)\mathcal{D}(U)\\
& \leq 2 C^* \frac{\eps^2}{\lambda}-\frac{1}{2}\mathcal{D}(U)\leq 0,
\end{aligned}
\end{align*}
which gives the desired result.

\vskip0.2cm
\noindent{\it Step 2:}  We now assume the other alternative, \textit{i.e.}, $\mathcal{D}(U)\leq 4 C^* \frac{\eps^2}{\lambda}.$ \\
We will use Proposition \ref{prop_near} to get the desired result.  First of all, we recall the constant $K$ satisfying \eqref{YC2} 
and the fixed small constant $\delta_1$ of Proposition \ref{prop_near} associated to the constant $K$.

Since $\mathcal{Y}_g(U)=\mathcal{Y}_g(\bar U)$, we have
$$
\mathcal{Y}_g(\bar U)=\mathcal{Y}(U)-\mathcal{Y}_b(U)-\mathcal{Y}_l(U)-\mathcal{Y}_s(U).
$$
Thus we have
$$
|\mathcal{Y}_g(\bar U)|^2\leq 4\left(|\mathcal{Y}(U)|^2+ |\mathcal{Y}_b(U)|^2+|\mathcal{Y}_l(U)|^2+|\mathcal{Y}_s(U)|^2\right),
$$
which can be written as
$$
-4|\mathcal{Y}(U)|^2\leq -|\mathcal{Y}_g(\bar U)|^2+ 4|\mathcal{Y}_b(U)|^2+4|\mathcal{Y}_l(U)|^2+4|\mathcal{Y}_s(U)|^2.
$$
Below, we will take $\delta_0$ small enough compared to the fixed constant $\delta_1$ (\textit{e.g.}  $\delta_0\le C\delta_1^9$). 
Then, using the facts that $\mathcal{B}_2^I(U)=\mathcal{B}_2^I(\bar U)$, $\mathcal{G}_2(\bar U)\le \mathcal{G}_2(U)$ and $\mathcal{D}(\bar U)\le \mathcal{D}(U)$, we find that for sufficiently small $\delta_0$ 
and for any $\delta_0^{-1}\eps<\lambda<\delta_0$,
\begin{align*}
\begin{aligned}
\mathcal{R}(U)&\le-\frac{4|\mathcal{Y}(U)|^2}{\eps\delta_1}+ \mathcal{B}(U)+\delta_0\frac{\eps}{\lambda} |\mathcal{B}(U)|-\mathcal{G}_1^I(U)- \mathcal{G}_1^O(U)- \mathcal{G}_2(U)
-(1-\delta_0)\mathcal{D}(U)\\
&\leq -\frac{|\mathcal{Y}_g(\bar U)|^2}{\eps\delta_1}+\left(\mathcal{B}_1(\bar U)+\mathcal{B}_2^I(\bar U)\right)+\delta_0\frac{\eps}{\lambda}\left(|\mathcal{B}_1(\bar U)|+|\mathcal{B}_2^I(\bar U)|\right)\\
&\quad -\left(1-\delta_1\frac{\eps}{\lambda}\right)\mathcal{G}_2(\bar U)-(1-\delta_1)\mathcal{D}(\bar U)\\
&\quad \underbrace{+\frac{4}{\eps\delta_1}\left(|\mathcal{Y}_b(U)|^2+|\mathcal{Y}_l(U)|^2+|\mathcal{Y}_s(U)|^2\right)}_{=:J_1} \\
&\quad \underbrace{+\left(1+\delta_0\frac{\eps}{\lambda}\right)\left(|\mathcal{B}_1(U)-\mathcal{B}_1(\bar U)|+|\mathcal{B}_2^O(U)| +|\mathcal{B}_3(U)|  \right)}_{=:J_2} \\
&\quad  \underbrace{-\mathcal{G}_1^I(U)-\mathcal{G}_1^O(U)-\delta_1\frac{\eps}{\lambda}\mathcal{G}_2(\bar U) -(\delta_1-\delta_0) \mathcal{D}(U)}_{=:J_3}.
\end{aligned}
\end{align*}
We claim that $J_1+J_2+J_3\le 0 $ for sufficiently small $\delta_0>0$. 
Indeed, it follows from \eqref{n1} and \eqref{m1} that for sufficiently small $\delta_0$ 
 and  for any $\eps/\lambda<\delta_0
 $, we have 
\begin{align*}
\begin{aligned}
J_1&\le \frac{C}{\delta_1}\frac{\eps}{\lambda}\left(\sqrt{\frac{\eps}{\lambda}}\mathcal{D}(U)+\mathcal{G}_1^O(U) + \left(\frac{\lambda}{\eps}\right)^{1/4}\mathcal{G}_1^I(U) +\left(\frac{\eps}{\lambda}\right)^{1/4}\mathcal{G}_2(\bar U) \right)\\
&\le \frac{C}{\delta_1}\left(\frac{\eps}{\lambda}\right)^{1/4} \left(\mathcal{D}(U)+\mathcal{G}_1^O(U) + \mathcal{G}_1^I(U) +\frac{\eps}{\lambda}\mathcal{G}_2(\bar U) \right)\\
&\le \frac{1}{4}\delta_1 \left(\mathcal{D}(U)+\mathcal{G}_1^O(U) + \mathcal{G}_1^I(U) +\frac{\eps}{\lambda}\mathcal{G}_2(\bar U) \right)  
\end{aligned}
\end{align*} and
\begin{align*}
\begin{aligned}
J_2&\le C 
\delta_0^{1/3}
 \left( \mathcal{D}(U) +\frac{\eps}{\lambda} \mathcal{G}_2(\bar U)\right) \le \frac{1}{4}\delta_1\left( \mathcal{D}(U) +\frac{\eps}{\lambda} \mathcal{G}_2(\bar U)\right).
\end{aligned}
\end{align*}
Therefore, if $\delta_0>0$ is small enough, then we have $J_1+J_2+J_3\le 0 $. Thus we get  
\begin{align*}
\begin{aligned}
\mathcal{R}(U)&\leq -\frac{|\mathcal{Y}_g(\bar U)|^2}{\eps\delta_1}+\left(\mathcal{B}_1(\bar U)+\mathcal{B}_2^I(\bar U)\right)+\delta_1\frac{\eps}{\lambda}\left(|\mathcal{B}_1(\bar U)|+|\mathcal{B}_2^I(\bar U)|\right)\\
&\quad -\left(1-\delta_1\frac{\eps}{\lambda}\right)\mathcal{G}_2(\bar U)-(1-\delta_1)\mathcal{D}(\bar U).
\end{aligned}
\end{align*}
Since the above quantities $\mathcal{Y}_g(\bar U), \mathcal{B}_1(\bar U), \mathcal{B}_2^I(\bar U), \mathcal{G}_2(\bar U)$ and $\mathcal{D}(\bar U)$ depends only on $\bar n$ through $\bar U$ and
we have $\mathcal{B}_1(\bar{U})=\mathcal{I}_1(\bar{n})$ and $0\leq \mathcal{B}_2^I(\bar{U})\le\mathcal{I}_2(\bar{n})$, it follows from Proposition \ref{prop_near} that $\mathcal{R}(U)\le 0$ (or see \eqref{near_ineq}). \\
Hence we complete the proof of Proposition \ref{main_prop}.\\

\begin{appendix}
\setcounter{equation}{0}

\section{Existence of the shift} \label{appendix_shift}
In this subsection, we present the existence of the shift satisfying \eqref{def_shift}
in Subsection \ref{subsec_existence_shift}. For a fixed $\vep>0$ and for a given solution $U\in\mathcal{X}_T$, we define 
$F:[0,T]\times\bbr\rightarrow\bbr$ by
$$F(t,X)= \Phi_\eps (\mathcal{Y}(U^X)) \Big(2|\mathcal{I}^{bad}(U^X)|+1 \Big)$$
where  $\Phi_\eps$ is as in \eqref{def_Phi} and  $\mathcal{Y}$ and $\mathcal{I}^{bad}$ are as in \eqref{badgood}.\\

We observe that  $\Phi_\eps, a$, $\tiln$ and $(1/\tiln)$ are bounded, $ \tilde n', \tilde n'', \tilde q',$ and $a'$ are   bounded and integrable. Together with  the information from  $U\in\mathcal{X}_T$,   
  we get  $\mathcal{D}(U)\in L^1(0,T)$ where
  $\mathcal{D}$ is defined in \eqref{d_only_a}. From these information, we can show
 \be\label{bdd_F} |F(t,x)|\leq C(1+\sqrt{\mathcal{D}(U^x)(t)})\quad\mbox{for} \quad t\in[0,T]\quad \mbox{and for } x\in\bbr.\ee Since we have $\sup_{x \in\bbr  }\mathcal{D}(U^x)(t)\leq C(\mathcal{D}(U)(t)+1)$ for each $t\in[0,T]$ and $\mathcal{D}(U)\in L^1(0,T)$, we can estimate $$\sup_{x\in\bbr }|F(t,x)|\leq a(t)$$ for some $a\in L^2(0,T)$.
 
 Similarly,  we can prove $$ 
\sup_{x \in\bbr  }|(D_xF)(t,x)|\leq b(t)
\quad \mbox{for } t\in[0,T] $$ for some function $ b\in L^2(0,T)$. Indeed, we can use the same idea as in \eqref{move-X} in order to move the translation symbol $(\cdot)^X$ from $U$ into smooth functions such as $a,\tilu$ and so on. It enables us to differentiate $F(t,x)$ with respect to $x$ without requiring any higher regularity of $U$. Then we can get a similar control for $|(D_xF)
|$ as in \eqref{bdd_F}. \\

 Then we can use the following lemma which is a simple adaptation of the well-known Cauchy-Lipschitz theorem.
 
 \begin{lemma}\label{lem_cauchy_lip}
 Let $p>1$ and $T>0$. Suppose that a function 
 $F:[0,T]\times\bbr\rightarrow\bbr$  satisfies 
 $$\sup_{x\in\bbr }|F(t,x)|\leq a(t) \quad\mbox{and}\quad
\sup_{x,y\in\bbr,x\neq y }\Big|\frac{F(t,x)-F(t,y)}{x-y}\Big|\leq b(t)
\quad \mbox{for } t\in[0,T] $$ for some functions $a \in L^1(0,T)$ and $\, b\in L^p(0,T)$. Then for any $x_0\in\bbr$, there exists a unique absolutely continuous function $X:[0,T]\rightarrow \bbr$ satisfying
\be\label{ode_eq}\left\{ \begin{array}{ll}
        \dot X(t) = F(t,X(t))\quad\mbox{for \textit{a.e.} }t\in[0,T],\\
       X(0)=x_0\end{array} \right.\ee
 \end{lemma}
 \begin{proof}
 First we note that \eqref{ode_eq} is equivalent to
\be\label{ode_equiv}X(t)=x_0+\int_0^t F(s,X(s))ds \quad\mbox{for  }t\in[0,T].\ee
 Then, the proof follows the classical   Picard's iteration argument:
 \ben\label{picard}\left\{ \begin{array}{ll}
 x_0(t)=x_0,\\
 x_{n+1}(t)=x_0+\int_0^t F(s,x_n(s))ds\quad \mbox{for}\quad n\geq 0\end{array} \right.\een Indeed, we observe that $a\in L^1$ makes the iteration possible. In particular,  $x_n:[0,T]\rightarrow \bbr$ is continuous and it satisfies $$\|x_n-x_0\|_{L^\infty(0,T)}\leq \|a\|_{L^1(0,T)}\quad \mbox{ for each } n.$$ 
Thanks to $b\in L^p$ with $p>1$, we take $t_*>0$ such that  $\|b\|_{L^p(0,T)}\cdot (t_*)^{1-(1/p)}\leq \frac{1}{2}$ and $t_*\leq T$. Then  we get, for each $n\geq1$,
   \begin{align*}
\begin{aligned}\label{lip_comp}
\|x_{n+1} -x_n\|_{L^\infty(0,t_*)}&\leq 
\int_0^{t_*}|F(s,x_n(s))-F(s,x_{n-1}(s))|ds 
 \leq  \|x_{n} -x_{n-1}\|_{L^\infty(0,t_*)}\cdot\int_0^{t_*} b(s)ds\\ & 
\leq  \|x_{n}-x_{n-1}\|_{L^\infty(0,t_*)}\cdot \|b\|_{L^p(0,T)}\cdot (t_*)^{1-(1/p)}\leq  \frac{1}{2}\cdot \|x_{n}-x_{n-1}\|_{L^\infty(0,t_*)}.
\end{aligned}
\end{align*} Thus we obtain
$\|x_{n+1} -x_n\|_{L^\infty(0,t_*)}\leq 2^{-n}\|a\|_{L^1(0,T)}$ so that
the uniform limiting function $X:[0,t_*]\rightarrow\bbr$ of the sequence $\{x_n:[0,t_*]\rightarrow \bbr\}_{n=1}^{\infty}$ exists and it satisfies \eqref{ode_equiv} for every $t\in[0,t_*]$. If $t_*<T$, then we just do  the process again with new data $X(t_*)$ in order to obtain $X$ on $[t_*,2t_*]$. Since we can repeat as many times as we want, we get $X$   up to the given time $T$. Similarly, uniqueness follows the assumption $p>1$.

 \end{proof}

\end{appendix}
  

\bibliography{ckkv2019bib}
 
\end{document}